\documentclass[11pt]{amsart}
\usepackage[top=4.2cm, bottom=4cm, left=2.4cm, right=2.4cm]{geometry}
\usepackage[utf8]{inputenc}
\usepackage[USenglish]{babel}
\usepackage[T1]{fontenc} 
\usepackage{mathrsfs}
\usepackage{wasysym}
\usepackage{tikz-cd}
\usepackage{mathtools}
\usepackage{amsmath}
\usepackage{amssymb,epsfig}
\usepackage{amsthm}
\usepackage[absolute]{textpos}
\usepackage{mathtools,emptypage}

\mathtoolsset{showonlyrefs}  

\usepackage{todonotes}
\usepackage{quiver}

\usepackage[bookmarks=true]{hyperref}
\usepackage{xcolor}
\hypersetup{
    colorlinks,
    linkcolor={red!50!black},
    citecolor={blue!50!black},
    urlcolor={blue!80!black},
}
\newcommand{\avint}{{\mathop{\,\rlap{-}\!\!\int}\nolimits}}

\newcommand{\ppi}{{\mbox{\boldmath$\pi$}}}
\newcommand{\sfd}{{\sf d}}
\newcommand{\restr}[1]{\lower3pt\hbox{$|_{#1}$}}
\newcommand{\ca}[1]{\overset{\circ}{{#1}}}

\newcommand{\Kliminf}{K\kern-3pt-\kern-2pt\mathop{\rm lim\,inf}\limits}  
\renewcommand{\d}{{\mathrm d}}

\newcommand{\e}{{\rm{e}}}                          
 \newcommand{\X}{{\rm X}}
 \newcommand{\Y}{{\rm Y}}

\newcommand{\limi}{\varliminf}
\newcommand{\lims}{\varlimsup}

\newcommand{\lip}{{\rm lip}}

\renewcommand{\div}{{\rm div}}

\newcommand{\mm}{\mathfrak m}                                

\newtheorem{theorem}{Theorem}[section]

\newtheorem{lemma}[theorem]{Lemma}
\newtheorem{proposition}[theorem]{Proposition}

\theoremstyle{definition}
\newtheorem{remark}[theorem]{Remark}
\newtheorem{definition}[theorem]{Definition}

\newtheorem{example}[theorem]{Example}

\title[Sobolev, BV and perimeter extensions in metric measure spaces]{Sobolev, BV and perimeter extensions\\ in metric measure spaces}
\author{Emanuele Caputo}
\author{Jesse Koivu}
\author{Tapio Rajala}

\address{University of Jyvaskyla \\
         Department of Mathematics and Statistics \\
         P.O. Box 35 (MaD) \\
         FI-40014 University of Jyvaskyla \\
         Finland}
         
\email{emanuele.e.caputo@jyu.fi}
\email{jesse.j.j.koivu@jyu.fi}
\email{tapio.m.rajala@jyu.fi}

\begin{document}

\begin{abstract}
 We study extensions of sets and functions in general metric measure spaces. We show that an open set has the strong BV extension property if and only if it has the strong extension property for sets of finite perimeter. We also prove several implications between the strong BV extension property and extendability of two different non-equivalent versions of Sobolev $W^{1,1}$-spaces and show via examples that the remaining implications fail.
\end{abstract}

\date{\today}
\thanks{The authors acknowledge the support from the Academy of Finland, grant no. 314789. The first named author also thanks the support from Academy of Finland, grant no. 321896}

\maketitle
\tableofcontents
\section{Introduction}
In this paper we study connections between the extendability of $BV$-functions, $W^{1,1}$-functions and of sets of finite perimeter in the setting of general metric measure spaces $(\X,\sfd,\mm)$ where the metric space $(\X,\sfd)$ is assumed to be complete and separable and the reference measure $\mm$ to be a nonnegative Borel measure which is finite on bounded sets.
More precisely, we study variants of the following question with different (subsets of) function spaces and (semi)norms: given an open set $\Omega \subset \X$ does there exist a constant $C>0$ such that for every $u \in BV(\Omega)$ there is $Eu\in BV(\X)$ with $Eu|_\Omega = u$ and $\|Eu\| \le C\|u\|$?
Sobolev spaces have been recently studied more and more in the  general context of metric measure spaces.
The generality will force us to find new ideas for proofs. In order to highlight this, we will next contrast our results and proofs with the more traditional settings for analysis on metric measure spaces.

After a series of fundamental works (in particular \cite{HK1998,C1999,HK2000}), the typical starting assumptions for questions that require more structure on the metric measure space are the validity of a local Poincar\'e inequality and a doubling property for the measure. Spaces satisfying these two assumptions are referred to as PI-spaces. When dealing with $W^{1,1}$- or BV-functions, the relevant Poincar\'e inequality is the $(1,1)$-Poincar\'e inequality, which allows one to control the $L^1$-norm of a function by the $L^1$-norm of its gradient.
One way the PI-assumption helps is that one can modify functions via partitions of unity to become locally Lipschitz so that the BV or Sobolev norm does not increase more than by a constant. We will return to this at the end of the introduction.
Another way the PI-assumption is used is to obtain compactness of bounded sets in the BV space with respect to $L^1$ topology, which in general might fail, see Remark \ref{rmk:compactness} and Example \ref{ex:BVfail}.
A consequence of the failure of the compactness is that the following restatement of the result of Burago and Maz'ya \cite{BM1967} (see also \cite[Section 9.3]{Maz2011}), although valid  on PI-spaces 
as observed by Baldi and Montefalcone \cite[Theorem 3.3]{BMF08}, fails in general, see Example \ref{ex:BVfail}.
\begin{theorem}[Burago and Maz'ya]\label{thm:BM}
A domain $\Omega \subset \mathbb R^n$ is a $\ca{BV}$-extension domain if and only if $\Omega$ has the extension property for sets of finite perimeter.
\end{theorem}

The definitions of $\ca{BV}$-extension and perimeter extension are given in Section \ref{sec:extension_properties}. These definitions take into account only the variation of the function (or the perimeter of the set) and not the $L^1$-norm of the function (nor the measure of the set).
In the Euclidean setting with a bounded domain, having extension with the full norm (the sum of the total variation and  the $L^1$-norm) is the same as having it with just the total variation part \cite{KMS2010}.
Notice, however, that if in the Euclidean space we drop the connectedness assumption (that is, consider just an open set instead of a domain), the two definitions of extendability do not agree. A simple example of this is the union of two disjoint disks in the plane. This has the extension property with the full norm but it does not have the extension property with just the total variation part. In the metric measure space setting without a PI-assumption the above difference between the extendability is also present even for domains.
Since having a domain instead of an open set will not make a difference in our setting, we will state our results for open sets.
In general metric measure spaces we are able to prove the following version of Theorem \ref{thm:BM}.

\begin{theorem}\label{thm:BV0}
 An open subset $\Omega \subset \X$ is a $(\ca{BV}\cap L^\infty,\|\cdot\|_{\ca{BV}})$-extension set if and only if it has the extension property for sets of finite perimeter.
\end{theorem}
The reason why in Theorem \ref{thm:BV0} 
we need to restrict to $L^\infty$-functions is because without a PI-assumption we cannot control the $L^1$-norms of the extensions even locally. One way to impose sufficient control on the $L^1$-norms is to take the definitions of extensions with respect to the full norms:
\begin{theorem}\label{thm:BV}
 An open subset $\Omega \subset \X$ is a $BV$-extension set if and only if it has the extension property for sets of finite perimeter with the full norm.
\end{theorem}

The connection between $W^{1,1}$-extensions and BV-extensions was studied by Garc\'{\i}a-Bravo and the third named author in \cite{BR21}. There a crucial role was played by the strong versions of $BV$- and perimeter extensions. In these versions, one requires the extension $Eu$ to give zero variation measure to the boundary of $\Omega$. See again Section \ref{sec:extension_properties} for the precise definitions. The statement from \cite{BR21} that we will generalize here is the following. 
\begin{theorem}[Garc\'{\i}a-Bravo and Rajala {\cite[Thm.\ 1.3]{BR21}}]\label{thm:GR}
Let $\Omega \subset \mathbb R^n$ be a bounded domain. Then the following are equivalent:
\begin{enumerate}
    \item $\Omega$ is a $W^{1,1}$-extension domain.
    \item $\Omega$ is a strong $BV$-extension domain.
    \item $\Omega$ has the strong extension property for sets of finite perimeter.
\end{enumerate}
\end{theorem}
Similarly to Theorem \ref{thm:BM}, also Theorem \ref{thm:GR} fails in general metric measure spaces. The reason is the same: failure of suitable compactness, and the counterexample is the same, Example \ref{ex:BVfail}. Therefore, we will state our result here with the full norm, analogously to Theorem \ref{thm:BV}. However, there are two other issues that arise in the general metric measure space setting. Firstly, the boundary of a Sobolev extension does not have in general measure zero. Recall that in PI-spaces, the measure density of extension domains holds and it implies via a density point argument that the boundary has measure zero \cite{HKT2008a,HKT2008b}. Secondly, there are several definitions of $W^{1,1}$ in metric measure spaces. Some of those definitions are not equivalent \cite{Ambrosio-DiMarino14} and for some the equivalence is still open.

We will state our results for two definitions of $W^{1,1}$. One definition is given via $\infty$-test plans, see Definition \ref{def:W11}. We denote the space of Sobolev functions given by this definition simply by $W^{1,1}(\X)$. The second definition we consider is $W_w^{1,1}(\X)$ which consists of $u \in BV(\X)$ for which $|D u|\ll \mm$. The third and the most studied definition would be the Newtonian Sobolev space $N^{1,1}(\X)$. Since we are aware of a work in progress where the equivalence of $N^{1,1}(\X)$ and $W^{1,1}(\X)$ will be shown, we will not separately consider extensions with respect to $N^{1,1}(\X)$, but only remark that in our results one can replace $W^{1,1}(\X)$ by $N^{1,1}(\X)$ and $W^{1,1}(\Omega)$ by $N^{1,1}(\Omega)$ once this equivalence is proven.

For an open set $\Omega \subset \X$ let us consider the following claims:
\begin{enumerate}
    \item[(s-Per)] $\Omega$ has the strong extension property for sets of finite perimeter with the full norm.
    \item[(s-BV)] $\Omega$ has the strong $BV$-extension property.
    \item[($W^{1,1}$)] $\Omega$ has the $W^{1,1}$-extension property.
    \item[($W_w^{1,1}$)] $\Omega$ has the $W_w^{1,1}$-extension property.
\end{enumerate}
Under the assumption that the boundary of the open set has measure zero, we have the full equivalence between the above properties.
\begin{theorem}\label{thm:boundaryzero}
Let $\Omega \subset \X$ be open and bounded with $\mm(\partial\Omega)=0$. 
Then
\[
(\text{s-Per}) \Longleftrightarrow (\text{s-BV}) \Longleftrightarrow (W^{1,1}) \Longleftrightarrow (W_w^{1,1}).
\]
\end{theorem}

If the boundary of the open set has positive measure, it might happen that the open set has the $W_w^{1,1}$-extension property, but not the $W^{1,1}$-extension property (nor the strong BV-extension property), see Example \ref{ex:slit}.
Moreover, an open set can have the $W^{1,1}$-extension property without having the strong $BV$-extension property, see Example \ref{ex:W11nonsBVcheat} and Example \ref{ex:W11nonsBV}.
The remaining implications excluded by the above examples do hold:

\begin{theorem}\label{thm:general}
Let $\Omega \subset \X$ be open and bounded. 
Then
\[
(\text{s-Per}) \Longleftrightarrow (\text{s-BV}) \Longrightarrow (W^{1,1}) \Longrightarrow (W_w^{1,1}).
\]
\end{theorem}

In the proofs of Theorem \ref{thm:boundaryzero} and Theorem \ref{thm:general} we need to change a BV-function into a Sobolev one without changing the boundary values of the function nor increasing the norm by much. In the proof of Theorem \ref{thm:GR} in \cite{BR21} this was done via a smoothing operator that was constructed using a Whitney decomposition and a partition of unity. In our proofs this approach does not work since a direct use of a partition of unity would require the Poincar\'e inequality. Instead, we make the modification individually for each function using the converging Lipschitz-functions given by the definition of the BV-space, see Proposition \ref{prop:smoothing}.

\section{Preliminaries and notations}
We assume throughout all this presentation that $(\X,\sfd,\mm)$ is a metric measure space, so that $(\X,\sfd)$ is a complete and separable metric space and $\mm$ in a nonnegative Borel measure which is finite on bounded sets.

Given a center point $x \in \X$ and a radius $r >0$, we denote the open ball by $B(x,r):= \{ y \in \X: \sfd(x,y) < r\}$.
We denote by $\mathscr{B}(\X)$ the collection of Borel subsets of $\X$, $\chi_A$ the indicator function of a set $A$ and $\mathcal{L}^d$ the Lebesgue measure on $\mathbb{R}^d$.
Given a set $A\subset \X$ and $r >0$, we denote the open $r$-neighbourhood of $A$ by $B(A,r):=\bigcup_{x \in A} B(x,r)$.
We recall the definition of the \emph{slope} $\lip f\colon \X \to \mathbb{R}^+$ of a function $f \colon \X \to \mathbb{R}$ given by 
\[ \lip f(x):= \lims_{y \to x} \frac{|f(y)-f(x)|}{\sfd(x,y)}\]
with the convention that $\lip f(x) = 0$ if $x \in \X$ is an isolated point. 

We denote by $\mathcal{L}^0(\mm)$ the space of $\mm$-measurable functions. Given $p \in [1,+\infty)$, we set \( \mathcal{L}^p(\mm):=\{ f \in \mathcal{L}^0(\mm): \int |f|^p\,d \mm < \infty \}\) and \( \mathcal{L}^\infty(\mm):= \{ f \in \mathcal{L}^\infty(\mm): \mm\textrm{-esssup}\, f < \infty\} \). We denote $L^p(\mm):= \mathcal{L}^p(\mm) /_{\sim}$ for $p \in \{0\} \cup [1,\infty]$, where $\sim$ is the equivalence relation given by $\mm$-a.e.\ equality.
We denote, for $p \in [1,\infty]$, $L^p_{\textrm{loc}}(\X):= \{ f \in L^0(\mm):\, \forall x\in \X\text{ there exists an open set }U \ni x\text{ s.t.\ }f \in L^p(\mm\restr{U})\}$. Similarly, we say, given $f_n,f\in L^p_{\textrm{loc}}(\X)$, that $f_n \to f\in L^p_{\textrm{loc}}(\X)$, provided that for every $x\in\X$, there exists an open set $U \ni x$ such that $f_n \to f\in L^p(\mm\restr{U})$.

Given an open set $U\subset X$, we define \[
\Gamma(U):=C([0,1],U)= \{ \gamma \colon [0,1] \to U,\, \gamma \text{ is continuous}\}
\]
which is a separable metric space when endowed with the sup distance. In the case in which $U = \X$ the space $\Gamma(\X)$ is also complete. We define, for $p \in [1,\infty]$, the set of $p$-absolutely continuous curves, $AC^p([0,1],U) \subset \Gamma(U)$ consisting of all $\gamma \in \Gamma(U)$ for which there exists $0 \le g \in L^p([0,1])$ such that
\[ \sfd(\gamma_t,\gamma_s) \le \int_t^s g(r)\,\d r\qquad\text{for every }0\le t \le s \le 1.\]
In this case, 
\[ |\gamma'_t|:=\lim_{h\rightarrow 0}\frac{\sfd(\gamma_{t+h},\gamma_t)}{|h|} \qquad \text{exists for a.e. }t \in [0,1]\]
and $|\gamma'_t| \in L^p([0,1])$.\\
We recall the definition of the evaluation map $\e_t \colon \Gamma(U) \to U$ as $\e_t(\gamma):=\gamma_t$, and note that it is continuous.
We denote by ${\rm Lip}(\gamma)$ the global Lipschitz constant of a  curve $\gamma \in AC^{\infty}([0,1],U)$
\[
{\rm Lip}(\gamma) = \sup_{t \ne s}\frac{\sfd(\gamma(t),\gamma(s))}{|t-s|}.
\]
Given a metric space $\Y$, we denote by $\mathscr{P}(\Y)$ the set of Borel probability measures on $\Y$. We recall the definition of a $\infty$-test plan (see \cite{Ambrosio-DiMarino14}).
\begin{definition}[Test plan on $U$]
A measure $\ppi \in \mathscr{P}(\Gamma(U))$ is a $\infty$-test plan if:
\begin{itemize}
    \item[i)] $\ppi$ is concentranted on $AC^{\infty}([0,1],U)$ and $(\gamma \mapsto {\rm Lip}(\gamma)) \in L^\infty(\ppi)$;
    \item[ii)] there exists ${\sf C} = {\sf C}(\ppi)$ such that ${\e_t}_*\ppi \le {\sf C} \mm$ for every $t \in [0,1]$.
\end{itemize}
\end{definition}
We call ${\sf C}$ the compression constant of $\ppi$ and we define ${\rm Lip}(\ppi):= \| {\rm Lip}(\gamma)\|_{L^\infty(\ppi)}$. Given an open set $U \subset \X$, for any $s,t \in [0,1]$ with $s<t$, we define the restriction map ${\rm restr}_{s,t} \colon C([0,1],U) \to C([0,1],U)$ as ${\rm restr}_{s,t}(\gamma)_r:= \gamma_{(1-r)s+rt}$ for $r \in [0,1]$. Notice that ${\rm restr}_{s,t}$ is continuous. A set $\Gamma \subset \Gamma(\X)$ is said to be $1$-negligible, if $\ppi(\Gamma) =0$ for every $\infty$-test plan $\ppi$. A property holds $1$-a.e.\ if the set where it does not hold is $1$-negligible.

\subsection{BV functions and sets of finite perimeter in metric measure spaces}
We define the space of functions of bounded variation.
\begin{definition}[Total variation]
Let $(\X,\sfd,\mm)$ be a metric measure space. Consider $f \in L^1_{\rm loc}(\X)$. Given an open set $A \subset \X$, we define
\begin{equation*}
    |D f|(A):= \inf \left\{ \limi_n \int_A \lip f_n\, \d \mm:\, f_n \in {\rm Lip}_{\rm loc}(A),\, f_n \to f \in L^1_{\rm loc}(\mm \restr{A}) \right\}. 
\end{equation*}
\end{definition}
We extend $|D f|$ to all Borel sets as follows: given $B \in \mathscr{B}(\X)$, we define
\[ |D f|(B):= \inf \left\{ |D f|(A), B \subset A, A \text{ is an open set}\right\}. \]
With this construction, $|D f| \colon \mathscr{B}(\X) \to [0,\infty)$ is a Borel measure, called the \emph{total variation measure} of $f$ (\cite[Thm.\ 3.4]{Mir03}).
It follows from the definition of total variation that, given an open set $A \subset \X$
\begin{equation}
    \label{eq:lsc_totvar_open}
    f_n \to f \quad \text{in }L^1_{\textrm{loc}}(A) \qquad \Rightarrow \qquad |D f|(A) \le \limi_{n \to \infty} |D f_n|(A).
\end{equation}
Given a Borel set $B\subset \X$ and $u \in L^1_{\textrm{loc}}(B)$, we introduce the notation $|D u|_B$ to mean the total variation of $u$ computed in the metric measure space $(\X,\sfd,\mm\restr{B})$.
\begin{definition}[$\ca{BV}(B)$ and $BV(B)$]
Let $(\X,\sfd,\mm)$ be a metric measure space. Let $B \subset \X$ be Borel. 
Given $u \in L^1_{\textrm{loc}}(\mm\restr{B})$, we define the space $\ca{BV}(B)$ to be set of functions $u \in L^1_{\textrm{loc}}(\mm\restr{B})$ for which $|D u|_B(B)<\infty$. We define $BV(B):=\{ u \in L^1(\mm\restr{B}):\,|D u|_B(B)<\infty \}$.
We endow the space $BV(B)$ with the norm $\| u \|_{BV(B)}:=\| u \|_{L^1(\mm)}+|D u|_B(B)$. Similarly, we endow the space $\ca{BV}(B)$ with the seminorm $\| u \|_{\ca{BV}(B)}:= |D u|_B(B)$.
\end{definition}
\begin{remark}
Notice that in the case of $B$ being open, the definition of $\ca{BV}(B)$ above is equivalent to saying that, given $f\in L^1_{\textrm{loc}}(\mm\restr{B})$, $|D\tilde{f}|(B) <\infty$, where $\tilde{f} \in L^1(\mm)$ is given by the zero extension and $|D\tilde{f}|(B) = |Df|_B(B)$.
Similarly, $f \in BV(B)$ if the last property holds and $f \in L^1(\mm\restr{B})$.
\end{remark}
%
%
%
\begin{remark}
We point out that, in the case of $f \in L^1(\mm)$, it is possible to define $|D f|(A)$ for an open set $A \subset \X$ by means of a relaxation with respect to the $L^1$-topology, namely
\[ |D f|(A):= \inf \left\{ \limi_n \int_A \lip f_n\, \d \mm:\, f_n \in {\rm Lip}_{\textrm{loc}}(A),\, f_n \to f \in L^1(\mm \restr{A}) \right\}. \]
\end{remark}
As a consequence of the lower semicontinuity of total variation, it can be readily checked that, given an open set $\Omega \subset \X$, $(BV(\Omega),\|\cdot\|_{BV(\Omega)})$ is a Banach space.
\begin{remark}
\label{rem:composition_lipschitz_functions}
Let $f \in \ca{BV}(\X)$ and $\varphi \colon \mathbb{R} \to \mathbb{R}$ be $L$-Lipschitz. Then, by the definition of total variation, we have $\varphi \circ f \in \ca{BV}(\X)$ and
\[ |D(\varphi \circ f)|(\X) \le L |D f|(\X). \]
\end{remark}
In particular, it follows by the definition that, if $U$ is open and bounded and $f$ is Lipschitz, then $f \in BV(U)$ with
\begin{equation}
\label{eq:bound_tot_variation_fLip}
    |D f|(U) \le \int_U \lip f\,\d \mm.
\end{equation}
\begin{definition}[Sets of finite perimeter]
We say that $E\in \mathscr{B}(\X)$ is a set of finite perimeter if $\chi_E \in \ca{BV}(\X)$ and we denote $P(E,B):=|D \chi_E|(B)$ for $B \in \mathscr{B}(\X)$, which is called \emph{the perimeter of $E$ in B}.
\end{definition} 
In particular, we call $P(E,\X)=:P(E)$ the perimeter of $E$.
We list here some useful properties of the perimeter. The validity of i),iii),iv) follows from the definition of perimeter and ii) by a diagonal argument (see \cite[Prop.\ 3.6]{Mir03}).
\begin{proposition}[Properties of the perimeter]
\label{prop:per_properties}
Let $(\X,\sfd,\mm)$ be a metric measure space. Consider two sets of finite perimeter $E$ and $F$. Let $U \subset \X$ be open and let $B \subset \X$ be Borel. Then:
\begin{itemize}
    \item[i)] \emph{Locality}. If $\mm((E \Delta F) \cap B)=0$, then 
    \begin{equation}
    \label{eq:locality_perimeter}
        P(E,B) = P(F,B);
    \end{equation}
    \item[ii)] \emph{Lower semicontinuity}. For every open set $U \subset \X$, the function $E \mapsto P(E,U)$ is lower semicontinuous with respect to $L^1_{\textrm{loc}}(\mm)$-topology, namely if $\chi_{E_n} \to \chi_E$ in $L^1_{\textrm{loc}}(\mm)$, then $P(E,U) \le \limi_n P(E_n,U)$;
    \item[iii)] \emph{Submodularity}. It holds $P(E \cup F,B) + P(E \cap F,B) \le P(E,B) + P(F,B)$;
    \item[iv)] \emph{Complementation}. It holds $P(E ,B) = P(\X \setminus E,B)$;
\end{itemize}
\end{proposition}
\begin{remark}\label{rmk:compactness}
We remark that without a Poincar\'e inequality, we do not always have compactness for sets of finite perimeter in the sense that any sequence of sets of finite perimeter $\{ E_n\}_n$ with $\sup_n P(E_n)< \infty$ would have a subsequence converging in $L^1_{\textrm{loc}}(\mm)$ to a limit set $E_\infty$ with finite perimeter. 
In PI spaces this holds, see \cite[Thm.\ 3.7]{Mir03}.
\end{remark}

We recall the coarea formula for BV functions in the setting of abstract metric measure spaces, as proved in \cite{Mir03} for PI spaces. As remarked for instance in \cite{Ambrosio-DiMarino14}, the same formula holds in the setting of abstract metric measure spaces. 
\begin{proposition}[Coarea formula]
Let $(\X,\sfd,\mm)$ be a metric measure space and consider $f \in L^1_{\textrm{loc}}(\mm)$. Then for every Borel set $E$, the map $t \mapsto P(\{ f> t \},E)$ is Borel and 
\[ |D f|(E) = \int_{-\infty}^{+\infty} P(\{ f > t \},E)\,\d t.\]
In particular, if $f \in \ca{BV}(\X)$, then $\{ f>t \}$ has finite perimeter for $\mathcal{L}^1$-a.e.\ $t \in \mathbb{R}$; on the other side, if $\int_{-\infty}^{+\infty} P(\{ f > t \},E)\,\d t<\infty$, then $f \in \ca{BV}(\X)$.
\end{proposition}
We define the notion of sets of finite perimeter on a Borel subset $B \subset \X$.
\begin{definition}[Sets of finite perimeter on a Borel subset B]
    Let $(\X,\sfd,\mm)$ be a metric measure space and $B\subset \X$. We say that $E \in \mathscr{B}(B)$ has finite perimeter on $B$ if $P_{B}(E) <\infty$, where $P_{B}(E) := |D \chi_E|_B(B)$ (where the total variation is computed in the metric measure space $(\X,\sfd,\mm\restr{B})$).
    Moreover, we define for every Borel set $F$, $P_B(E,F):= |D \chi_E|_B(B \cap F)$.
\end{definition}
\begin{remark}
\label{rem:from_borel_to_open}
Again, as for the case of $BV$ functions, we notice that, if $B$ is an open set, $E\in \mathscr{B}(\X)$ has finite perimeter in $B$ if and only if $P(\tilde{E},B)<\infty$, where $\tilde{E}$ is any Borel set such that $(\tilde{E} \cap \Omega) \Delta E =\emptyset$. In this case, $P(\tilde{E},B) = P_B(E)$.
\end{remark}
\begin{remark}
\label{rem:coarea_formula_onBorel}
We have for the definitions of $\ca{BV}(B)$ and sets of finite perimeter on $B$ that the coarea formula in the mms $(\X,\sfd,\mm\restr{B})$ reads as follows.
Consider $f \in L^1_{\textrm{loc}}(\mm\restr{B})$.
Then, for every Borel set $E$, the map $t \mapsto P_B(\{ f> t \},E)$ is Borel and
\[ |D f|_B(E) = \int_{-\infty}^{+\infty} P_B(\{ f > t \},E)\,\d t.\]
In particular, if $f \in \ca{BV}(B)$, then $\{ f > t\}$ has finite perimeter on $B$ for $\mathcal{L}^1$-a.e.\ $t$; on the other side, if $\int_{-\infty}^{+\infty} P_B(\{ f > t \},\X)\,\d t<\infty$, then $f \in \ca{BV}(B)$.
\end{remark}
\subsection{Sobolev 
functions in metric measure spaces}
We recall that there are several possible definitions of $W^{1,1}$ in arbitrary metric measure spaces, 
see for instance \cite{Ambrosio-DiMarino14}. Let us consider an open set $\Omega \subset \X$.
The simplest definition after having defined $BV(\Omega)$ is the space $W_w^{1,1}(\Omega) \subset BV(\Omega)$ consisting of all $u \in BV(\Omega)$ such that $|D u|\ll \mm\restr{\Omega}$, endowed with the norm as subset of the BV space, namely:
\begin{equation}
    \| u\|_{W^{1,1}_w(\Omega)}:=\| u\|_{L^1(\Omega)}+\left\|\frac{\d |D u|}{\d \mm} \right\|_{L^1(\Omega)}, \qquad \text{for }u\in W^{1,1}_w(\Omega).
\end{equation}
In this paper we will not consider the Newtonian definition of Sobolev space \cite{KMM98,S00}. One reason for this is that it is not the most convenient one to use in our proofs. 
Instead, the main definition of Sobolev space for the exponent $p = 1$ that we use in this paper is the following.
\begin{definition}[The space $W^{1,1}(\X)$]\label{def:W11}
Given $f \in L^1(\mm)$, we say that $f \in W^{1,1}(\X)$ if there exists $G \in L^1(\mm)$ such that, for every $\infty$-test plan $\ppi$
\[ |f(\gamma_1) - f(\gamma_0) | \le \int_0^1 G(\gamma_t)|\dot{\gamma_t}|\,\d t\qquad\text{for }\ppi\text{-a.e. }\gamma.\]
In this case, we call $G$ a $1$-weak upper gradient of $f$.
\end{definition}
We can localize in time via the following standard argument (see \cite[Prop. 5.7]{AGS14}).
Given an $\infty$-test plan $\ppi$ and the fact that the probability measure $\ppi_{q_1,q_2}:={{\rm restr}_{q_1,q_2}}_*\ppi$ for $q_1,q_2 \in \mathbb{Q} \cap [0,1]$ is an $\infty$-test plan, writing the definition of $W^{1,1}(\X)$ checked on $\ppi_{q_1,q_2}$ and using the fact that, for $\ppi$-a.e.\ $\gamma$, $f \circ \gamma \in W^{1,1}(0,1)$, we get the following.
\begin{proposition}
\label{prop:equivalence_localized_W11}
Given $f \in L^1(\mm)$, the following are equivalent:
\begin{itemize}
    \item[i)] $G$ is a $1$-weak upper gradient of $f$
    \item[ii)] for every $\infty$-test plan $\ppi$, for $\ppi$-a.e.\ $\gamma$, $f \circ \gamma \in W^{1,1}(0,1)$ and
    \[ |(f \circ \gamma)'_t| \le G(\gamma_t)|\dot{\gamma}_t|\qquad \text{for a.e. }t.\]
\end{itemize}
\end{proposition}
As a consequence of the last proposition, we have that, defining
\[A(f):=\{ G \in L^1(\mm): G\text{ is a }1\text{-weak upper gradient of }f \},\]
$(A(f),\le)$ is a convex, closed (in $L^1(\mm)$) lattice. Hence there exists a $1$-weak upper gradient, which is minimal $\mm$-a.e., which we call the minimal $1$-weak upper gradient and denote by $|Df|_{1,\X}$.
Similarly, given an open set $\Omega \subset \X$, it is natural to define $W^{1,1}(\Omega)$ by considering only test plans on $\Omega$. We denote the 1-minimal weak upper gradient for this case $|Du|_{1,\Omega}$.\\
It is immediate to check that $W^{1,1}(\X) \subset W^{1,1}(\Omega)$ (where the inclusion is given by the natural restriction) and 
\begin{equation}
\label{eq:1mwug_on_Omega}
    |D u|_{1,\Omega} \le |D u|_{1,\X}\quad \mm\text{-a.e. on }\Omega \quad \text{for every }u \in W^{1,1}(\X).
\end{equation}
As customary, we eliminate the subscripts $1$ and $\Omega$ when there is no danger for confusion. 

Notice that the following inclusion holds
\begin{equation}
 W^{1,1}(\X) \subset W_w^{1,1}(\X),
\end{equation}
see \cite[Sec.\ 8]{Ambrosio-DiMarino14}.
Notice also that, as a consequence of the equivalence of definitions of the space $BV(\X)$ in \cite[Thm. 1.1]{Ambrosio-DiMarino14}, the definition by relaxation is equivalent to another one using the notion of $\infty$-test plans (see the definition of $w-BV(\X,\sfd,\mm)$ therein).
\begin{theorem}[Equivalent definition of $BV(\X)$ {\cite[Thm. 1.1]{Ambrosio-DiMarino14}}]\label{thm:BVequiv}
Let $f \in L^1(\mm)$. Then $f \in BV(\X)$ if and only if for $1$-a.e.\ curve we have that $f \circ \gamma \in BV(0,1)$ and $|f(\gamma_1)-f(\gamma_0)| \le |D(f \circ \gamma)|(0,1)$ and for every $\infty$-test plan $\ppi$
\begin{equation}
    \label{eq:curvewise_BV}
    \int \gamma_*|D(f \circ \gamma)|(B)\,\d \ppi(\gamma) \le {\sf C}(\ppi)\,\| {\rm Lip}(\gamma) \|_{L^\infty(\ppi)}\,\mu(B).
\end{equation}
    for every $B \in \mathscr{B}(\X)$.
In this case, $|Df|$ is the minimal measure $\mu$ for which \eqref{eq:curvewise_BV} is satisfied.
\end{theorem}

%
\subsection{Extension properties}
\label{sec:extension_properties}
We introduce some extension properties of a Borel set $U\subset \X$. Often the set $U$ will be assumed to be open.
We define the $(\mathscr{A},\|\cdot\|_{\mathscr{S}})$-extension set, where $\mathscr{A}(U) \subset L^0(\mm\restr{U})$ is a vector space, when endowed with a seminorm $\| \cdot \|_{\mathscr{S}(U)}$. For what concerns this manuscript, we will specialize the above definition in the case where $(\mathscr{A}(U),\|\cdot\|_{\mathscr{S}(U)})$ is one of the following: $(W^{1,1}(U),\|\cdot\|_{W^{1,1}(U)})$, $(W^{1,1}_w(U),\|\cdot\|_{W^{1,1}_w(U)})$, $(BV(U),\|\cdot\|_{BV(U)})$, $(BV(U)\cap L^\infty(U),\|\cdot\|_{BV(U)})$, $(\ca{BV}(U)\cap L^\infty(U),\|\cdot\|_{\ca{BV}(U)})$, or $(\ca{BV}(U),\|\cdot\|_{\ca{BV}(U)})$.
\begin{definition}
\label{def:A_extension_domains}
Let $\Omega \subset \X$ be a Borel set, $\mathscr{A}(\Omega) \subset L^0(\mm\restr{\Omega})$, and $\| \cdot \|_{\mathscr{S}(\Omega)}$ a seminorm on $\mathscr{A}(\Omega)$. Then we say that $\Omega$ is an $(\mathscr{A},\|\cdot\|_{\mathscr{S}})$-extension set if there exists $C >0$ and $E \colon \mathscr{A}(\Omega) \to \mathscr{A}(\X)$ such that 
\begin{itemize}
    \item[i)] $\|E u\|_{\mathscr{S}(\X)} \le C \| u\|_{\mathscr{S}(\Omega)}$ for every $u \in \mathscr{A}(\Omega)$;
    \item[ii)] $E u\restr{\Omega} = u$ for every $u \in \mathscr{A}(\Omega)$.
\end{itemize}
We will write that $\Omega$ is a $\mathscr{A}$-extension set instead of $(\mathscr{A},\|\cdot\|_{\mathscr{S}})$-extension set whenever $\|\cdot\|_{\mathscr{S}(\Omega)}$ is a natural seminorm on $\mathscr{A}(\Omega)$.
\end{definition}
If $\Omega$ is a $(\mathscr{A},\|\cdot\|_{\mathscr{S}})$-extension set with $E$ being the extension operator, we define:
\begin{equation}
    \| E \|_{\mathscr{S}}:= \sup_{u \in \mathscr{A}(\Omega)} \frac{\|E u\|_{\mathscr{S}(\X)}}{\| u \|_{\mathscr{S}(\Omega)}} <\infty
\end{equation}
with the convention that $0/0=0$ and $t/0 = \infty$ for $t>0$. Notice that this may happen since $\|\cdot\|_{\mathscr{S}(\Omega)}$ is in this generality only a seminorm.
%
%
%
%
\begin{definition}[Strong BV extension sets]
\label{def:strong_bv_ext_domain}
Let $\Omega \subset \X$ be open. Then we say that $\Omega$ is a strong $BV$-extension set (s-$BV$-extension set in short) if it is a $(BV,\| \cdot\|_{BV})$-extension set with extension operator $E$ and
\begin{equation}
\label{eq:no_charging_bdry_definition}
    |D E u|(\partial \Omega) = 0.
\end{equation}
\end{definition}
In all the definitions above, when $\Omega$ is also connected we say that $\Omega$ is a $(\mathscr{A},\|\cdot\|_{\mathscr{S}})$-extension domain and for the last definition a s-$BV$-extension domain in place of extension set. Analogous definitions of extendability can be given for sets of finite perimeter. For completeness and for fixing the terminology, we write these definitions below explicitly.
\begin{definition}[Extension property for sets of finite perimeter]
\label{def:ext_prop_set}
Let $\Omega \subset \X$ be a Borel set. Then we say that $\Omega$ has the extension property for sets of finite perimeter if there exists $C_{\textrm{Per}}>0$ such that for every $E \subset \Omega$ of finite perimeter on $\Omega$ there exists $\tilde{E}$ such that the following hold
\begin{itemize}
    \item[i)] $P(\tilde{E},\X) \le C_{\textrm{Per}} P_\Omega(E)$;
    \item[ii)] $\mm(E \Delta (\tilde{E} \cap \Omega)) = 0$; 
\end{itemize}
\end{definition}
\begin{definition}[Extension property for sets of finite perimeter for the full norm]
\label{def:ext_prop_set_fullnorm}
Let $\Omega \subset \X$ be a Borel set. Then we say that $\Omega$ has the extension property for sets of finite perimeter for the full norm if there exists $C_{\textrm{Per}}'>0$ such that for every $E \subset \Omega$ of finite perimeter in $\Omega$ there exists $\tilde{E}\subset \X$ such that the following hold
\begin{itemize}
    \item[i)] $\mm(\tilde{E}) + P(\tilde{E},\X) \le C_{\textrm{Per}}'( \mm(E)+P_{\Omega}(E))$;
    \item[ii)] $\mm(E \Delta (\tilde{E} \cap \Omega)) = 0$; 
\end{itemize}
\end{definition}
\begin{definition}[Strong extension property for sets of finite perimeter]
\label{def:strong_ext_prop}
Let $\Omega \subset \X$ be open. Then we say that $\Omega$ has the strong extension property for sets of finite perimeter if there exists $C_{\textrm{Per}}'>0$ such that for every $E \subset \Omega$ of finite perimeter in $\Omega$ there exists $\tilde{E}$ such that i) and ii) in Definition \ref{def:ext_prop_set_fullnorm} hold and
\begin{itemize}
    \item[iii)] $P(\tilde{E}, \partial \Omega) = 0$. 
\end{itemize}
\end{definition}
\section{Relations between extension properties for BV functions and for sets of finite perimeter} \label{sec:bv_vs_per}

In this section we will prove Theorem \ref{thm:BV0} and Theorem \ref{thm:BV} connecting extendability of $BV$-functions and sets of finite perimeter. Unlike in the Euclidean setting where a bounded domain is $BV$-extension set if and only if it is a $\ca{BV}$-extension set, \cite[Lemma\ 2.1]{KMS2010}, in general metric measure spaces that does not support a Poincar\'e inequality this need not be true.
As already mentioned in the introduction, a simple way to see this is to consider the union of two disjoint balls in Euclidean space. This set is a $BV$-extension set: we can consider the extensions from the balls separately and use partition of unity to make a global extension operator. However, the set is not a $\ca{BV}$-extension set: consider a function that is zero in one ball and one in the other. Then the extension should have zero gradient almost everywhere, which is impossible by the Poincar\'e  inequality. Although the union of two balls is not a domain, by considering a weight on the Euclidean space so that the capacity between the two balls is zero inside some domain $\Omega$ containing the two balls and nonzero in the whole space, we obtain a domain $\Omega$ in a metric measure space that is a $BV$-extension set but not a $\ca{BV}$-extension set. Since having a domain instead of an open set does not provide more analytic restrictions in the general setting, below we will not assume $\Omega$ to be a domain. Moreover, we will state in Proposition \ref{prop:weak_equivalence} and Proposition \ref{prop:fullnorm_equivalence} slightly more general versions of Theorem \ref{thm:BV0} and Theorem \ref{thm:BV} where the set $\Omega$ is assumed to be only Borel.

We will also give Example \ref{ex:BVfail} showing that being a $\ca{BV}$-extension set is not the same as having the extension property for sets of finite perimeter. This is due to the lack of compactness in $BV(\X)$ with respect to the total variation. At the end of the section we prove Proposition \ref{prop:strong_equivalence} showing the first equivalence in Theorem \ref{thm:boundaryzero} and Theorem \ref{thm:general}. Let us also mention that one can also prove the intermediate version between Proposition \ref{thm:BV0} and Proposition \ref{prop:strong_equivalence} showing the equivalence between strong $\ca{BV}\cap L^\infty$-extendability and strong extendability of sets of finite perimeter. The simple variation of the proofs is left to the interested reader.
%



We start with the proof of a slightly more general version of Theorem \ref{thm:BV0}.
%
%
\begin{proposition}
\label{prop:weak_equivalence}
 A Borel subset $\Omega \subset \X$ is a $(\ca{BV}\cap L^\infty,\|\cdot\|_{\ca{BV}})$-extension set if and only if it has the extension property for sets of finite perimeter.
 \end{proposition}
\begin{proof}
We first assume that the Borel set $\Omega \subset \X$ is a $(\ca{BV}\cap L^\infty,\|\cdot\|_{\ca{BV}})$-extension set and show that then $\Omega$ has the extension property for sets of finite perimeter. Consider a Borel set $S$ such that $P_{\Omega}(S)<\infty$. Then $\chi_S \in \ca{BV}(\Omega)\cap L^\infty(\Omega)$ and, by hypothesis, there exists $E \colon \ca{BV}(\Omega)\cap L^\infty(\Omega) \to \ca{BV}(\X) \cap L^\infty(\X)$ as in Definition \ref{def:A_extension_domains}.
Denoting $u:=E \chi_S$, we can assume without loss of generality that $0 \le u \le 1$, by considering $\varphi \circ E$ in place of  $E$, where $\varphi(t) = \max\{ \min\{ t,1\},0\}$; this is still a $\ca{BV}$-extension operator, as a consequence of Remark \ref{rem:composition_lipschitz_functions} (since $\varphi$ is $1$-Lipschitz), and $\| \varphi \circ E\|_{\ca{BV}} \le \| E\|_{\ca{BV}}$. 
By applying the coarea formula, we have 
\[
|D u|(\X)= \int_0^1 P(\{ u > t \} ,\X)\, \d t.
\]
Moreover, there exists $t_0 \in [0,1]$ such that $P(\{ u > t_0 \},\X) \le \int_0^1 P(\{ u>t \},\X)\,\d t$. We choose such $t_0$. Combining the last two facts, we have
\[ P(\{ u > t_0\},\X) \le \| E\|_{\ca{BV}} P_{\Omega}(S).\]
Therefore, the set $\tilde{S}:=\{ u > t_0 \}$ verifies items i) and ii) in Definition \ref{def:ext_prop_set} with the constant $C_{\textrm{Per}} = \|E\|_{\ca{BV}}$.
\medskip

Let us then prove the converse implication and assume that $\Omega$ has the extension property for sets of finite perimeter with a constant $C_{\textrm{Per}}>0$. Consider $u \in \ca{BV}(\Omega)\cap L^\infty(\Omega)$. First, we notice that we may assume without loss of generality that $-1 \le u \le 1$. Indeed, if we build the extension operator $E$ in such a case, we can consider in the general one $\tilde{E} u:=\| u\|_{L^\infty(\Omega)} E(u / \| u\|_{L^\infty(\Omega)})$ and notice that $\|\tilde{E}\|_{\ca{BV}} \le \| E\|_{\ca{BV}}$. We may also assume that $0 \le u \le 1$. Indeed, given $E$ for functions with such a property, for the previous case consider $\tilde{E} u := E u^+- E u^-$ and notice that $\|\tilde{E}\|_{\ca{BV}} \le \| E\|_{\ca{BV}}$.
By applying the coarea formula as in Remark \ref{rem:coarea_formula_onBorel} we know that there exists $N \subset [0,1]$ with $\mathscr{L}^1(N) = 0$ so that
\[
P_{\Omega}(\{ u > t \},\X)<\infty\qquad\text{for all }t \in [0,1] \setminus N.
\]
For every $t \in [0,1]\setminus N$, we extend the set $E_t := \{ u> t \}$ to a Borel set $\tilde{E}_t$ such that $P(\tilde{E}_t,\X) \le C_{\textrm{Per}} P_{\Omega}(E_t)$. Our goal is to find $v_n \in BV(\X)$ such that 
\[
|D v_n|(\X) \le (1+1/n) C_{\textrm{Per}} |D u|_{\Omega}(\Omega)
\]
for every $n$, such that $v_n \to v$ in $L^1_{\textrm{loc}}(\X)$ and $v = u$ on $\Omega$, which gives the conclusion by the application of \eqref{eq:lsc_totvar_open}. We define $u_n = \sum_{j=1}^{2^n} 2^{-n} \chi_{\tilde E_{t_j}}$ for some $t_j \in [(j-1)2^{-n},j2^{-n}]$, that will be chosen later.
We fix $n$ and define $\delta_n := 2^{n}(1+1/n)$. There exists a Borel set $I$ with $\mathscr{L}^1(I)>\delta_n^{-1}$ such that $I \subset [(j-1)2^{-n},j2^{-n}] =: I_{j,n}$ such that, for every $t \in I$, $P_{\Omega}(E_t) \le \delta_n \int_{I_{j,n}} P_{\Omega}(E_r)\,\d r$. For every $j$, choose $t_j$ in the set $I$ defined above.
Therefore, we compute
\begin{equation}
\label{eq:estimate_total_variation_weak_equivalence}
\begin{aligned}
    |D u_m|(\X) &\le \sum_{j=1}^{2^n} 2^{-n} P(\tilde{E}_{t_j},\X) \le C_{\textrm{Per}} \sum_{j=1}^{2^n} 2^{-n} P_{\Omega}(E_{t_j}) \\
    &\le \delta_n C_{\textrm{Per}} 2^{-n} \sum_{j=1}^{2^n} \int_{I_{j,n}} P_{\Omega}(E_r)\,\d r = C_{\textrm{Per}} \delta_n 2^{-n}  |D u|_{\Omega}(\Omega).
\end{aligned}
\end{equation}
We have that $u_n \to u$ in $L^1_{\textrm{loc}}(\Omega)$. Indeed, notice that, given $x \in E_{t_j} \setminus E_{t_{j+1}}$, we have $u_n(x) = (j-1) 2^{-n}$, hence $|u(x)-u_n(x)| \le 2^{-n}$; therefore, for every $x \in \Omega$ and $r >0$ we have $\| u_n-u \|_{L^1(B(x,r) \cap \Omega)} \to 0$ as $n \to \infty$. We consider the measure $\tilde{\mm} \in \mathscr{P}(\X)$ defined as 
\begin{equation*}
\tilde{\mm} := \sum_{i=1}^{\infty}\frac{\mm \restr{B(x_0,i)\setminus B(x_0,i-1)}}{\mm(B(x_0,i)\setminus B(x_0,i-1))2^i},
\end{equation*}
In particular, it holds that $\mm \ll \tilde{\mm} \ll \mm$.
Since $0\le u_n \le 1 \in L^2(X,\tilde{\mm})$ for every $n$, 
we have
\[
\sup_n \| u_n \|_{L^2(\X,\tilde{\mm})}<\infty.
\]
Hence there exists a subsequence $u_{n_k} \rightharpoonup v$ in $L^2(X,\tilde{\mm})$. We apply Mazur's lemma to this subsequence to obtain a sequence $v_m = \sum_{i=m}^{N_m} \lambda_{i}u_{n_{k_i}}$ such that $v_m \to v$ in $L^2(X,\tilde{\mm})$ as $m \to \infty$. We define $E u:= v$. The subadditivity of total variation and \eqref{eq:estimate_total_variation_weak_equivalence} give
\[
|Dv_m|(X)\leq (1+1/m)C_{\textrm{Per}} |Du|_{\Omega}(\Omega). 
\]
Therefore, we just need to show that $v=u$ $\mm$-a.e.\ on $\Omega$. It suffices to prove $v_m \to v$ as $m \to \infty$ in $L^1_{\textrm{loc}}(\X)$.
Indeed, consider $B := B(x_0,R)$ for $R >0$ and $x_0 \in \X$; we estimate
\begin{equation*}
\|v-v_m \|_{L^1(B)} \leq \sqrt{\mm(B)}\,\|v-v_m \|_{L^2(B)} \leq C\|v-v_m\|_{L^2(X,\tilde{\mm})}\to 0.
\end{equation*}
Here the final inequality follows since we can compare $\mm$ and $\tilde{\mm}$ on a bounded set, as it is contained in a finite union of the annuli $B(x_0,i)\setminus B(x_0,i-1)$. This shows that $\Omega$ is a $(\ca{BV} \cap L^\infty, \ca{BV})$-extension set with $\|E\|_{\ca{BV}} \le C_{\textrm{Per}}$.
\end{proof}

\begin{remark}
We remark that a $\ca{BV}$-extension set $\Omega$ is always a $\ca{BV}\cap L^\infty$-extension set. This is seen by extending a $u \in \ca{BV}(\Omega)\cap L^\infty(\Omega)$ first as a function $Eu \in \ca{BV}(\Omega)$ and then cutting it from below and above by the essential infimum and supremum of $u$ in $\Omega$. By Theorem \ref{thm:BV0}, we then conclude that $\Omega$ also has the extension property for sets of finite perimeter.
\end{remark}

\begin{example}\label{ex:BVfail}
We consider an example of a domain in a metric measure space which has the extension property for sets of finite perimeter, but it is not a ${\ca{BV}}$-extension domain.
Let us consider in $\mathbb{R}^2$ the following sequence of sets. We define for $k \in \mathbb{N}$, the sets
\[
T_k:= (1-2^{-k},1-2^{-(k+1)})\times (1-2^{-k},\mathbb{R}) \cup (1-2^{-k},\mathbb{R}) \times (1-2^{-k},1-2^{-(k+1)}).
\]
We consider as $\Omega$ the open triangle with vertices in the points $(0,0)$, $(0,1)$ and $(1,0)$. We define the auxiliary sets $S_k:= \partial T_k \cup (\partial \Omega \cap T_k)$ for $k \in \mathbb{N}$. Let us define the function $\rho \in L^1(\mathcal{L}^2)$ as $\rho = 1$ on $T_k$ whenever $k=2n$ with $n \in \mathbb{N}$ and $\rho =(2^{k+2} \sfd(\cdot,S_k))\wedge 1$, whenever $k=2n+1$ with $n \in \mathbb{N}$. Moreover, we extend $\rho$ to be $0$ outside. Let us consider the metric measure space $(\mathbb{R}^2,|\cdot|,\rho \mathcal{L}^2)$.
The main objects of the construction are represented in Figure \ref{fig:example_perext_notbvdotext}.

\begin{figure}[h]
    \centering
    \includegraphics[scale=0.4]{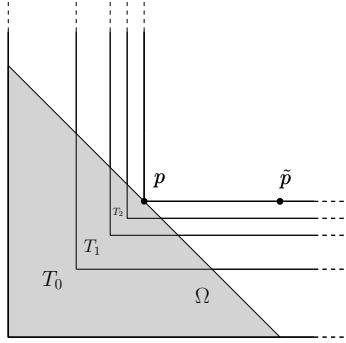}
    \caption{The domain $\Omega$ of Example \ref{ex:BVfail} having the extension property for sets of finite perimeter, but not for $\ca{BV}$. The reference measure is constructed so that a function in $\Omega$ may blow up when approaching the point $p$, but still have zero total variation. It is shown that there exists such function for which any extension with zero total variation fails to be locally integrable at the point $\tilde p$. }
    \label{fig:example_perext_notbvdotext}
\end{figure}

We claim that $\Omega$ has the extension property for sets of finite perimeter. To do so, let us work for a moment in the Euclidean setting and let us consider the set $T_0$ as defined above and $\Omega_0:=T_0\cap \Omega$. We have that $\Omega_0$ is a Lipschitz domain in $\mathbb{R}^2$ and thus a $BV$-extension domain with respect to the Lebesgue measure. Therefore, $\Omega_0$ is also a $\ca{BV}$-extension domain by \cite[Lemma 2.1]{KMS2010} with respect to the Lebesgue measure, and so it also has the extension property for sets of finite perimeter by Theorem \ref{thm:BM}.
Let us call $H\colon \{ \,\text{sets of finite perimeter in }\Omega_0\, \} \to \{ \,\text{sets of finite perimeter in }\mathbb{R}^2\, \}$ the extension operator. In particular, the operator 
\[F \colon \{ \,\text{sets of finite perimeter in }\Omega_0\, \} \to \{ \,\text{sets of finite perimeter in }T_0\, \}\]
defined as $F:= {\rm restr} \circ H$, where ${\rm restr}$ is the restriction operator to sets of finite perimeter in $T_0$ verifies
\begin{equation}
    \label{eq:continuity_for_F}
    {\rm Per}(F(E),T_0)\le C_{\Omega_0}\,{\rm Per}(E,\Omega_0).
\end{equation}
for some $C_{\Omega_0} >0$. Let us denote by $i_k\colon T_0 \to T_k$ the natural homothety rescaling.
Given $E \subset \Omega$ such that ${\rm Per}(E,\Omega)<\infty$,
we then define the following extension. Let the set $\tilde{F}(E)$ be defined as:
\begin{equation*}
\begin{array}{ll}
    \tilde{F}(E)\cap T_k= i_k(F(i_k^{-1}(E \cap T_k)))&\text{if }k=2n\text{ for }n \in \mathbb{N},\\
    \tilde{F}(E)\cap T_k= E \cap T_k&\text{if }k=2n+1\text{ for }n \in \mathbb{N}.
\end{array}
\end{equation*}
We now show $\tilde{F}$ is the extension operator for sets of finite perimeter. It follows from the definition of $\tilde{F}$ that $\tilde{F}(E) \cap \Omega = E$; moreover,
\begin{equation*}
    \begin{aligned}
     {\rm Per}(\tilde{F}(E))& = \sum_{k\text{ even}} {\rm Per}(\tilde{F}(E), T_k)=\sum_{k\text{ even}} {\rm Per}(i_k(F(i_k^{-1}(E \cap T_k))),i_k(T))\\
     &= \sum_{k\text{ even}} 2^{-k}\, {\rm Per}(F(i_k^{-1}(E \cap T_k)),T) \stackrel{\eqref{eq:continuity_for_F}}{\le} C_{\Omega_0} \,\sum_{k\text{ even}} 2^{-k} \,{\rm Per}(i_k^{-1}(E \cap T_k),i_k^{-1}(\Omega\cap T_k)) \\
     &=C_{\Omega_0}\,\sum_{k\text{ even}}\,{\rm Per}(E \cap T_k,\Omega\cap T_k) \le C_{\Omega_0} {\rm Per}(E,\Omega),
    \end{aligned}
\end{equation*}
thus concluding the claim.
We now prove that $\Omega$ is not a ${\ca{BV}}$-extension domain. To do so, let us define the function
\begin{equation*}
    u:=\left\{\begin{array}{lll} 2^k &\text{on }T_k \cap \Omega\text{ for }k=2n&\text{ with }n \in \mathbb{N},\\
    0 &\text{on }T_k\cap \Omega\text{ for }k=2n+1&\text{ with }n \in \mathbb{N}.\\
    \end{array}\right.
\end{equation*}
We check that $u \in L^1(\Omega,\rho\mathcal{L}^2)$. To do so, it is enough to check that, given a sufficiently small $r >0$, we have $u \in L^1(B_r(p))$ for $p=(1,1)$.
To do so, we compute
\begin{equation*}
    \| u \|_{L^1(B_r(p))} \le \sum_{k=0}^\infty 2^k |T_k \cap \Omega| = \sum_{k=0}^\infty 2^{k} 2^{-2k} |T_0| < |T_0|<\infty.
\end{equation*}
Assume that there exists an extension operator $\bar{F} \colon {\ca{BV}}(\Omega) \to {\ca{BV}}(\mathbb{R}^2,|\cdot|,\rho\mathcal{L}^2)$ and call $\tilde{u}:=\bar{F} u$. Then we would have that $|D \tilde{u}|(\mathbb{R}^2) \le C |D u|(\Omega)=0$, which gives that $|D \tilde{u}|(T_k)=0$ for every $k \in \mathbb{N}$. Thus, we can characterize $\tilde{u}$ as
\begin{equation*}
    \tilde{u}=\left\{\begin{array}{lll} 2^k &\text{on }T_k\text{ for }k=2n&\text{ with }n \in \mathbb{N},\\
    0 &\text{on }T_k\text{ for }k=2n+1&\text{ with }n \in \mathbb{N}.\\
    \end{array}\right.
\end{equation*}
The contradiction lies in the fact that $\tilde{u}\notin L^1_{\textrm{loc}}(\rho \mathcal{L}^2)$.
Indeed, let us consider the point $\tilde{p}=(2,1)$ and we take the cube centered at $\tilde{p}$ with sizes $r$ and we denote it by $Q_r(\tilde{p})$. We have 
\begin{equation*}
    \| \tilde{u} \|_{L^1(Q_r(\tilde{p}))} \ge \sum_{k\ge \bar{k}(r)} 2^k \mathcal{L}^2(T_k \cap Q_r(\tilde{p})) \ge C \sum_{k\ge \bar{k}(r)} 2^k\, 2^{-k} > \infty.
\end{equation*}
thus having a contradiction.
\end{example}
%
%

Next we will prove a slightly more general version of Theorem \ref{thm:BV}. 

\begin{proposition}
\label{prop:fullnorm_equivalence}
 A Borel subset $\Omega \subset \X$ is a $BV$-extension set if and only if it has the extension property for sets of finite perimeter with the full norm.
\end{proposition}
\begin{proof}
We prove the only if part. Thus, we assume $\Omega$ is a $BV$-extension set. Let $S\subset \Omega$ be such that $\mm(S) +P_{\Omega}(S)< \infty$ and consider $u := \chi_S \in BV(\Omega)$. Then, considering $E$ the extension operator given by the assumption, we have $E u \in BV(\X)$. Moreover, arguing as in the only if part of the proof of Proposition \ref{prop:weak_equivalence}, we may assume that $E u$ takes values in $[0,1]$ (further assuming that $\varphi$ defined therein satisfies $|\varphi(t)| \le |t|$ for $t \in \mathbb{R}$) and
\[ 
\int_0^1 \mm(\{ E u > t \})+P(\{ E u > t \},\X)\,\d t = \| E u \|_{BV(\X)} \le \| E \|_{BV}(\mm(S)+P_{\Omega}(S))
\]
where we applied firstly coarea formula together with Cavalieri's formula and then item i) in Definition \ref{def:A_extension_domains}.
We choose $t_0 \in I$ such that 
\begin{equation}
\mm(\{ E u > s\})+P(\{ E u >s \},\X) \le \int_0^1 \mm(\{ E u > t\})+P(\{ E u >t \},\X)\,\d t
\end{equation}
and we have that $\tilde{S}:= \{ E u >t_0 \}$ is the desired extension with the choice of the constant $C_\textrm{Per}'=\|E\|_{BV}$.

We then prove the if part. Thus, we assume $\Omega$ to have the extension property for sets of finite perimeter for the full norm.

{\color{blue} Step 1}. Firstly, we prove that $\Omega$ is a $(BV \cap L^\infty, \|\cdot\|_{BV})$-extension set. We follow the proof of the if part in Theorem \ref{thm:BV0}, stressing the main differences.
We consider $u \in BV(\Omega) \cap L^\infty(\Omega)$ and we assume without loss of generality, by similar arguments as before, that $0 \le u \le 1$ $\mm$-almost everywhere. By applying the coarea formula and the Cavalieri's identity, defining $E_t:=\{ u > t \}$, we have:
\begin{equation*}
    \int |u|\,\d \mm + |D u|_{\Omega}(\Omega) =   \int_0^1 (\mm(E_t)+P_{\Omega}(E_t))\,\d t.
\end{equation*}
In particular, $\mm(E_t)+P_{\Omega}(E_t) < \infty$ for $t \in [0,1] \setminus N$, with $\mathcal{L}^1(N) = 0$.
For such $t$'s, we define, by our assumption on extendability of sets of finite perimeter, a Borel set $\tilde{E}_t$ such that $\tilde{E}_t \cap \Omega = E_t$ a.e.\ on $\Omega$ and 
\begin{equation}
\label{eq:ext_full_norm_superlevelsets}
\mm(\tilde{E}_t)+P(\tilde{E}_t,\X) \le C_{\textrm{Per}}' (\mm(E_t)+P_{\Omega}(E_t)).
\end{equation}
We define $I_{j,n}$ and $\delta_n$ as before and we fix $n$. Given $j$, we have that there exists a Borel set $I \subset I_{j,n}$ with $|I|> \delta_n^{-1}$ such that for every $t \in I$ $\mm(E_t)+P_{\Omega}(E_t) \le \delta_n\int_{I_{j,n}}\mm(E_r)+P_{\Omega}(E_r)\,\d r$. We choose $t_j \in I \cap ([0,1] \setminus N)$ and we define $u_n =\sum_{j=1}^{2^n} 2^{-n} \chi_{\tilde{E}_{t_j}}$. We compute
\begin{equation*}
    \begin{aligned}
    & \| u_n \|_{L^1(\mm)} = 2^{-n} \sum_{j=1}^{2^n} \mm(\tilde{E}_{t_j}),\qquad |D u_n|(\X) \le 2^{-n} \sum_{j=1}^{2^n} P(\tilde{E}_{t_j},\X).\\
    \end{aligned}
\end{equation*}
Combining the last two inequalities and \eqref{eq:ext_full_norm_superlevelsets}, we have
\begin{equation*}
\begin{aligned}
    \| u_n \|_{L^1(\X)}+ |D u_n|(\X) &\le C_{\textrm{Per}}' 2^{-n} \sum_{j=1}^{2^n} (\mm(E_{t_j})+ P_{\Omega}(E_{t_j})) \le C_{\textrm{Per}}' (1+\frac{1}{n}) \int_0^1 \mm(E_t)+ P_{\Omega}(E_t)\,\d t \\
    & = C_{\textrm{Per}}' (1+\frac{1}{n}) (\| u \|_{L^1(\Omega)}+ |D u|_{\Omega}(\Omega)). \\
\end{aligned}
\end{equation*}
Hence we get that $\| u_n\|_{BV(\X)} \le C_{\textrm{Per}}'(1+1/n) \| u\|_{BV(\Omega)}$. Arguing as before with the help of Mazur's lemma, we can find $v_n \in BV(\X)$ such that $\|v_n\|_{BV(\X)} \le \| u_n\|_{BV(\X)}$, $\|v_n\|_{L^\infty(\X)}\le 1$, $v_n \to v$ in $L^1_{\textrm{loc}}(\X)$, $v_n \to u$ in $L^1_{\textrm{loc}}(\Omega)$. Hence $v=u$ $\mm$-a.e.\ on $\Omega$ and, using the lower semicontinuity of the map $L^1_{{\rm loc}}(\X) \ni f \mapsto \| f \|_{L^1(\X)}+|D f|(\X)$, we have $\|v\|_{BV(\X)} \le C_{\textrm{Per}}'\| u\|_{BV(\Omega)}$.

{\color{blue} Step 2.} To conclude, we prove that, if $\Omega$ is a $(BV \cap L^\infty,\|\cdot\|_{BV})$-extension set, then it is a $BV$-extension set. Indeed, let $u \in BV(\Omega)$. For every $i$, define $\varphi_i \colon \mathbb{R} \to \mathbb{R}$ as
\[
 \varphi_i(t) = \begin{cases} -1, & \text{if }t < -i-1,\\
 t+i,&\text{if }-i-1 \le t \le -i,\\
 0, &\text{if }-i<t<i,\\
 t-i,&\text{if }i \le t \le i+1,\\
 1,&\text{if }i+1 < t.
 \end{cases}
\]
It is straightforward to check that $\sum_{i=0}^\infty \varphi_i =1$ on $\mathbb{R}$. Moreover, defining $u_i=\varphi_i \circ u$, we have that $u=\sum_{i=0}^{\infty} u_i$ and, by means of Cavalieri's formula and coarea formula, it holds that $\| u\|_{BV(\Omega)}=\sum_{i=0}^{\infty} \| u_i\|_{BV(\Omega)}$. Fix $i$; since $u_i \in BV(\Omega)\cap L^\infty(\Omega)$, by assumption we know that there exists $v_i \in BV(\X)$ such that $\| v_i\|_{BV(\X)} \le C_{\textrm{Per}}' \| u_i\|_{BV(\Omega)}$. We define $v:=\sum_{i=0}^\infty v_i$ and we notice that $v =u$ $\mm$-a.e.\ on $\Omega$ and
\begin{equation}
    \| v\|_{BV(\X)} \le \sum_{i=0}^\infty \|v_i\|_{BV(\X)} \le C_{\textrm{Per}}' \sum_{i=0}^\infty \|u_i\|_{BV(\Omega)} = C_{\textrm{Per}}' \| u \|_{BV(\Omega)},
\end{equation}
thus concluding the proof.
\end{proof}
We turn now to the equivalence for strong extension properties. 
We state these results only for open sets $\Omega$.
The if part in this case needs a modification of the argument, in spirit of the recent work \cite{BR21}.
In the proof, we will use the following known proposition.
\begin{lemma}
\label{lemma:lusin_uniform_convergence}
Let $g_n \colon [0,1] \to \mathbb{R}$ be an increasing (or decreasing) sequence of measurable functions pointwise converging to $g \colon [0,1] \to \mathbb{R}$. For every $\varepsilon>0$, there exists a compact set $K \subset [0,1]$ such that $\mathcal L^1([0,1] \setminus K) \le \varepsilon$ for which $g_n \to g$ uniformly on $K$.
\end{lemma}
\begin{proposition}
\label{prop:strong_equivalence}
Let $\Omega$ be an open set. Then $\Omega$ is a strong $BV$-extension set if and only if it has the strong extension property for the sets of finite perimeter with the full norm.
\end{proposition}
\begin{proof}
 We first prove the only if part. We repeat the arguments of the proof of Proposition \ref{prop:fullnorm_equivalence}, with $Eu \in BV(\X)$ defined by assumption. By \eqref{eq:no_charging_bdry_definition} and coarea formula 
\[
0 =|D u|(\partial \Omega) = \int_0^1 P(\{ u > t \},\partial \Omega)\,\d t,
\]
and thus $P(\{ u> t\},\partial \Omega) = 0$ for a.e.\ $t$. Choosing $t_0$ also outside of this exceptional set, we get that $\tilde{S}:=\{ u > t_0\}$ verifies items i)-iii) of Definition \ref{def:strong_ext_prop}.
\medskip

Let us then prove the if part. 

{\color{blue} Step 1}. Firstly, we prove that $\Omega$ is a $(BV \cap L^\infty, \|\cdot\|_{BV})$-extension set and given $E$ the extension operator, it holds $|D (E u)|(\partial\Omega)=0$. We again repeat the arguments of the if part in the proof of Proposition \ref{prop:weak_equivalence}, pointing out the differences. We assume without loss of generality that $0 \le u \le 1$ and define $E_t := \{ u > t\}$. By the coarea formula, we know that $E_t$ has finite perimeter in $\Omega$, so we define $\tilde{E}_t$ satisfying the assumptions of Definition \ref{def:strong_ext_prop} and in the negligible set we define $\tilde{E}_t:= E_t$. We apply Lemma \ref{lemma:lusin_uniform_convergence} with $g_k(t) := P(\tilde{E}_t,B(\partial \Omega,2^{-k}))$ and $g(t) := 0$ and we define $K_m$ as the set given by the lemma for $\varepsilon = 2^{-m}$.
We define $I_{j,n}$, $\delta_n$ as before and fix $n$. We choose $\tilde{I}_i^n\subset {I}_i^n$ as in the if part of the proof of Proposition \ref{prop:weak_equivalence}.
For every $i$, we choose $t_i^n \in I_i^n \cap K_m$ if $I_i^n \cap K_m \neq \emptyset$, otherwise $t_i^n \in I_i^n$.
We define $u_n:= \sum_{i=1}^{2^{n}} 2^{-n} \chi_{\tilde{E}_{t_i^n}}$.
By uniform convergence of $t \mapsto P(E_t,B(\partial \Omega,2^{-k}))$ to $0$ on $K_m$, we know that for every $\varepsilon>0$ there exists $\bar{k} = \bar{k}(\varepsilon)$ such that $P(E_t,B(\partial \Omega,2^{-k})) \le \varepsilon$ for $k \ge \bar{k}$.
We define $S:=\{ i: \tilde{I}_i^n \cap K_m = \emptyset \} \subset \{1,\dots,2^{n}\}$ and
\[ w(\delta):= \sup \left\{ \int_A \mm(E_s)+P(E_s,\Omega)\, \d s:\, A \subset [0,1],\,\mathcal{L}^1(A) = \delta \right\}.\]
We compute
\begin{equation*}
\begin{aligned}
|D u_n|(B(\partial \Omega, 2^{-\bar{k}})) &\le \sum_{i=1}^{2^n} 2^{-n}P(\tilde{E}_{t_i^n},B(\partial \Omega, 2^{-\bar{k}}))\\
& = \sum_{i \in S^c} 2^{-n}P(\tilde{E}_{t_i^n},B(\partial \Omega, 2^{-\bar{k}})) + \sum_{i \in S} 2^{-n}P(\tilde{E}_{t_i^n},B(\partial \Omega, 2^{-\bar{k}}))\\
& \le \sum_{i \in S^c} 2^{-n}P(\tilde{E}_{t_i^n},B(\partial \Omega, 2^{-\bar{k}})) + \sum_{i \in S} 2^{-n}P(\tilde{E}_{t_i^n},\X)\\
& \le \varepsilon + (1+1/n)\,C_{\textrm{Per}}'\, \sum_{i\in S} \int_{\tilde{I}_i^n}\mm(E_t)+P(E_t,\Omega)\,\d t \\
&\le \varepsilon + (1+1/n)\,C_{\textrm{Per}}'\,\int_{\cup_{i \in S} \tilde{I}_i^n} \mm(E_t)+P(E_t,\Omega)\,\d t  \\
& \le \varepsilon + (1+1/n)\,C_{\textrm{Per}}'\,\int_{K_m^c} \mm(E_t)+P(E_t,\Omega)\,\d t \le \varepsilon + (1+1/n)\, C_{\textrm{Per}}' w(2^{-m}),
\end{aligned}
\end{equation*}
where the inclusion $\bigcup_{i \in S}\tilde{I}_{i}^n \subset K_m^c$ follows by the definition of $S$. We choose $\varepsilon = 2^{-m}$.
By applying again Mazur's lemma as in the proof of Proposition \ref{prop:fullnorm_equivalence}, we can find $v_n \in BV(\X)$ such that $\|v_n\|_{BV(\X)} \le \| u_n\|_{BV(\Omega)}$, $\|v_n\|_{L^\infty(\X)}\le 1$, $v_n \to v$ in $L^1_{\textrm{loc}}(\X)$, $v_n \to u$ in $L^1_{\textrm{loc}}(\Omega)$. Hence $v=u$ $\mm$-a.e.\ on $\Omega$ and $\| v \|_{BV(\X)} \le C_{\textrm{Per}}' \| u \|_{BV(\Omega)}$.
Moreover, it holds that $|D v_n|(B(\partial \Omega, 2^{-\bar{k}}))\le |D u_n|(B(\partial \Omega, 2^{-\bar{k}}))$.
Hence, we have
\[ |D v|(\partial \Omega) \le |D v|(B(\partial \Omega, 2^{-\bar{k}})) \le 2^{-m} + 2 C_{\textrm{Per}}' w(2^{-m}),\]
where in the last inequality we used
\eqref{eq:lsc_totvar_open} applied to the open set $B(\partial \Omega, 2^{-\bar{k}})$. By taking the limit as $k \to +\infty$, we get that $|D v|(\partial \Omega) = 0$, thus concluding the proof.

{\color{blue} Step 2.} We consider $u \in BV(\Omega)$; we can argue similarly to Step 2 in the proof of Proposition \ref{prop:fullnorm_equivalence} and notice that for the functions $u_i$ we can apply the conclusions of Step 1 of this proposition and call in analogy $v_i$ the extensions; define $v$ accordingly. The conclusion holds by following the arguments of the proof of Proposition \ref{prop:fullnorm_equivalence} together with the inequality $|D v|(\partial \Omega) \le \sum_{i=0}^{\infty} |D v_i|(\partial \Omega) =0$.
\end{proof}
\begin{remark}
We point out that in the proofs of the statements in Proposition \ref{prop:weak_equivalence}, Proposition \ref{prop:fullnorm_equivalence} and Proposition \ref{prop:strong_equivalence} the constants in the extensions are the same constants given by the respective assumptions. This is the idea between the choice of $\delta_n$ in the proofs of these propositions.
\end{remark}
\section{Relations between Sobolev extension domains and BV extension domains}

This section is divided in two parts. In Section \ref{sec:smoothing}, we present a smoothing argument, which is the core idea to prove the main theorems in Section \ref{sec:main_propositions}, relating $W^{1,1}$ and strong $BV$ extension sets. In the final part, some examples are presented, showing in particular the sharpness of the assumption that $\mm(\partial \Omega) >0$ in Proposition \ref{prop:from_wW11_to_sBV}.
All the implications and examples are summarized in Figure \ref{fig:diagram}.
%
\subsection{Smoothing argument}

\label{sec:smoothing}
Here we prove a smoothing argument, which is the main tool we use to relate the notions of $W^{1,1}$ and strong $BV$ extension sets. Another smoothing argument in the Euclidean setting was presented in \cite[Thm.\ 3.1]{BR21} using a Whitney decomposition. That approach gave a linear smoothing operator, but required the use of a Poincar\'e inequality. Here our operator is not linear, but it works without the Poincar\'e inequality.
\begin{proposition}[Smoothing operator]
\label{prop:smoothing}
Let $\Omega \subset \X$ be open. There exists a constant $C$ such that the following holds: for every $\varepsilon >0$, there exists $T_\varepsilon \colon BV(\Omega) \to {\rm Lip}_{\textrm{loc}}(\Omega)$ such that
\begin{equation}
\label{eq:smoothing_normbounds}
\|T_\varepsilon u\|_{L^1(\Omega)} \le \varepsilon + \|u\|_{L^1(\Omega)},\quad
\int_{\Omega} \lip\,T_\varepsilon u\,\d \mm \le C(|D u|(\Omega) +\varepsilon),
\end{equation}
$T_\varepsilon u-u \in BV(\X)$ (when defined to be $0$ in $\X \setminus \Omega$) and
\begin{equation}
    \label{eq:notchargingtheboundary}
    |D(T_\varepsilon u - u )|(\partial \Omega) = 0.
\end{equation}
\end{proposition}
Before going into the proof, let us outline the main idea. Given $\epsilon >0$ and $u \in BV(\Omega)$, we define the smoothing $T_\epsilon u$ as follows. Firstly, we consider a partition of unity subordinated to strips which are thinner close to the boundary of $\partial \Omega$; then, we fix a strip and consider an approximating sequence for the total variation on the strip and select a function in the sequence with sufficiently large index. Finally, we sum up the selected functions using the partition of unity.
Then \eqref{eq:notchargingtheboundary} 
follows by considering larger indexes in the approximations on the strips, and by building a sequence of locally Lipschitz functions $(\psi_k)_k$ converging in $L^1(\Omega)$ to $u-T_{\epsilon} u$ as $k \to \infty$.
Finally, a first order control on $T_\epsilon u$ can be pointwisely estimated by considering an auxiliary sequence of locally Lipschitz approximations of $u$ in $\Omega$ for its total variation. This leads to \eqref{eq:smoothing_normbounds}.
\begin{proof}
We consider $\Omega_0:= \{ \sfd(\cdot,\Omega^c) > 2^{-1} \}$ and $\Omega_i:=\{ 2^{-(i+1)} < \sfd(\cdot,\Omega^c) < 2^{-(i-1)} \}$ for $i \in \mathbb{N}$ and notice that $\Omega = \bigcup_{i=0}^\infty \Omega_i$. 
Notice that as a consequence of the definition of the sets $\{\Omega_i\}_i$, in particular of the fact that $\Omega_i \cap \Omega_{i=2} = \emptyset$ for every $i \in \mathbb{N} \cup \{ 0 \}$, we have:
\begin{equation}
\label{eq:sum_tv_strips}
    \sum_{i=0}^\infty |D u|(\Omega_i) = \sum_{i=0, i\text{ even}}^\infty |D u|(\Omega_i)+\sum_{i=0, i\text{ odd}}^\infty |D u|(\Omega_i) \le 2 |D u|(\Omega).
\end{equation}
We consider $\{\varphi_i\}_{i \in \{0\} \cup \mathbb{N}}$ to be a Lipschitz partition of unity subordinated to the covering $\{\Omega_i\}_{i \in \{0\} \cup \mathbb{N}}$ with the properties 
\begin{equation}
    \varphi_i=0\text{ on }\Omega_i^c,\, \lip \varphi_i \le 2^{i+1}\chi_{\Omega_i}\text{ for every }i\quad\text{ and }\quad\sum_{i=0}^{\infty} \varphi_i =1.
\end{equation}
Indeed, such a family can be easily constructed as follows. We define on $(0,\infty)$ the family of functions
\begin{equation}
    r_0(t):= \begin{cases} 0, & \text{if }t < 2^{-1},\\
 2(t-1/2),&\text{if }2^{-1} \le t \le 1,\\
 1, &\text{if }t>1.
 \end{cases}\qquad r_i(t):= \begin{cases} 
 2^{i+1}(t-2^{-(i+1)}),&\text{if }\text{if }2^{-(i+1)} \le t \le 2^{-i},\\
 2^{i}(2^{-(i-1)}-t), & \text{if }2^{-i} \le t \le 2^{-(i-1)},\\
 0, &\text{if }t \le 2^{-(i+1)}\text{ or }t \ge 2^{-(i-1)}
 \end{cases}\quad 
\end{equation}
for every $i \ge 1$ and notice that $\sum_{i=0}^\infty r_i =1$ on $(0,\infty)$. Then, we define $\varphi_i :=r_i(\sfd(\cdot,\Omega^c))$ and we have $\lip\,\varphi_i \le (\lip\,r_i)(\sfd(\cdot,\Omega^c)) \le 2^{i+1}\chi_{\Omega_i}$, using in the first inequality that $\sfd(\cdot,\Omega^c)$ is $1$-Lipschitz. The remaining properties can be readily checked.

Consider $u \in BV(\Omega)$, so its restrictions belong to $BV(\Omega_i)$ for every $i$. By definition, there exists a sequence of $u_n^i \in {\rm Lip}_{\textrm{loc}}(\Omega_i)$ such that $u_n^i \to u$ in $L^1(\Omega_i)$ and $\int_{\Omega_i} \lip\, u_n^i\,\d \mm \to |D u|(\Omega_i)$.
We consider $n_i$ such that for every $j \ge n_i$ we have $\| u_{j}^i-u \|_{L^1(\Omega_i)} \le \varepsilon 2^{-2i}$ and
\begin{equation*}
\int_{\Omega_i} \lip \,u_j^i\,\d \mm   \le \begin{cases}
 2 |D u|(\Omega_i) & \text{whenever }|D u|(\Omega_i)>0,\\
 \varepsilon \,2^{-i}& \text{whenever }|D u|(\Omega_i)=0.\\
\end{cases}
\end{equation*}
In particular, for every $i$, we have the trivial bound
\begin{equation}
\int_{\Omega_i} \lip \,u_j^i \le 2 |D u|(\Omega_i) + \varepsilon \,2^{-i}\,\text{ for every }j \ge n_i.
    \label{eq:bound_lip_approx}
\end{equation}
We define $u_i:= u^i_{n_i}$ for every $i$. Moreover, for every $i$, there exists $m_{i,k} \in \mathbb{N}$ such that, for every $j \ge m_{i,k}$, $\| u_{j}^i-u \|_{L^1(\Omega_i)}\le \varepsilon 2^{-2i-k}$.
We define $T_\varepsilon u$ as the function $\tilde{u}:= \sum_{i=0}^{\infty} \varphi_i u_i $. Then, $\tilde{u} \in L^1(\Omega) \cap {\rm Lip}_{\textrm{loc}}(\Omega)$ and
\[ \|\tilde{u}\|_{L^1(\mm)} \le \|\tilde{u}-u\|_{L^1(\mm)}+\|u\|_{L^1(\mm)} = \|\sum_{i=0}^\infty \varphi_i (u_i-u)\|_{L^1(\mm)}+ \|u\|_{L^1(\mm)} =\varepsilon + \|u\|_{L^1(\mm)}.\]
We define $\psi_k:= \sum_{i=0}^{k-1} \varphi_i u^i_{\max{\{n_i,m_{i,k}}\}} - \sum_{i=0}^{k-1} \varphi_i u_i$, where $\tilde{\varphi}_k:=\sum_{i \ge k} \varphi_i$ and $u$, $\tilde{u}$ and $\varphi$ are meant with zero extension outside of $\Omega$. 
We check that $\psi_k \to u-\tilde{u}$ in $L^1(\mm)$. We compute
\begin{equation*}
    \begin{aligned}
        \psi_k-(u-\tilde{u}) = \sum_{i=1}^{k-1} \varphi_i (u^i_{\max{\{n_i,m_{i,k}}\}}-u)+\sum_{i=k}^{\infty} \varphi_i (u_i-u).
    \end{aligned}
\end{equation*}
By estimating the $L^1$-norm on both sides and recalling that $\varphi_i$ is supported on $\Omega_i$ we have:
\begin{equation*}
\begin{aligned}
    \| \psi_k-(u-\tilde{u}) \|_{L^1(\Omega)} &\le \sum_{i=1}^{k-1} \|u^i_{\max{\{n_i,m_{i,k}}\}}-u\|_{L^1(\Omega_i)}+\sum_{i=k}^{\infty} \|u_i-u\|_{L^1(\Omega_i)} \\
    &\le \varepsilon 2^{-k}\sum_{i=1}^{k-1} 2^{-2i}+ \varepsilon \sum_{i=k}^\infty 2^{-2i} \le \varepsilon 2^{-(k-1)}.\\
\end{aligned}
\end{equation*}

We prove \eqref{eq:notchargingtheboundary}. Fix $\delta >0$. We consider $\bar{k} = \bar{k}(\delta)$ so that $\delta \in [2^{-(\bar{k}+1)},2^{-\bar{k}}]$; in particular, we have that $\bar{k} \to \infty$ when $\delta \to 0$. Since $\psi_k \to u-\tilde{u}$ on $B(\partial \Omega,\delta)$, by the definition of $|D(u-\tilde{u})|$ on open sets, we have:
\begin{equation*}
\begin{aligned}
    |D (u-\tilde{u})|(B(\partial \Omega,\delta)) & \le \limi_{k \to \infty} \int_{B(\partial \Omega,\delta)} \lip\,\psi_k\,\d\mm \\
    & \le \limi_{k \to \infty}\sum_{i=\bar{k}}^{k} \int \lip \varphi_i\, (|u^i_{\max{\{n_i,m_{i,k}}\}}-u|-|u_i-u|)+ \varphi_i\,(\lip \,u^i_{\max{\{n_i,m_{i,k}}\}}+\lip \,u_i)\,\d \mm \\
    & \stackrel{\eqref{eq:bound_lip_approx}}{\le} \limi_{k \to \infty} \sum_{i=\bar{k}}^{\infty} \left( 2^i (\|u^i_{\max{\{n_i,m_{i,k}}\}}-u\|_{L^1(\Omega_i)}+\|u_i-u\|_{L^1(\Omega_i)})+ 4|D u|(\Omega_i)+ 2 \varepsilon 2^{-i}\right)\\
    & \stackrel{\eqref{eq:sum_tv_strips}}{\le} \limi_{k \to \infty} \sum_{i=\bar{k}}^{k} 2^{i} 2^{-2i-k}+\sum_{i=\bar{k}}^{\infty} 2^{i} 2^{-2i}+8 |D u|(B(\partial \Omega,\delta) \cap \Omega)+4\varepsilon \delta\\
    &\le 4 \delta+8|D u|(B(\partial \Omega,\delta) \cap \Omega)+ 4 \varepsilon \delta.\\
\end{aligned}
\end{equation*}
We point out that, in the last two lines, we use that, by the choice of $\bar{k}$, $2^{-\bar{k}}\le 2 \delta$. We take the limit as $\delta \to 0$ and  conclude that $| D(u-\tilde{u})|(\partial \Omega)=0$.
Moreover, since we know that $u-\tilde{u} \in BV(U)$, provided $U$ is open and $U \subset X \setminus \bar{\Omega}$ or $U \subset \Omega$, we get that $u-\tilde{u} \in BV(\X)$.
It is left to prove the second inequality in \eqref{eq:smoothing_normbounds}, namely we check that there exists $C>0$ such that $\int_\Omega \lip\,\tilde{u}\,\d \mm \le C (|D \tilde{u}|(\Omega)+\varepsilon)$; in particular, this inequality grants that $\lip \tilde{u}$ is a 1-weak upper gradient of $u$, hence $\tilde{u } \in W^{1,1}(\Omega)$.
To do so, it is enough to show that there exists $C>0$ such that, for every $m$, $\int_{\cup_{i=0}^m \Omega_i} \lip\, u\,\d \mm \le C (|Du|(\Omega)+ \varepsilon)$. Let us prove it.
We consider a sequence $\tilde{u}_m$ such that $\| u-\tilde{u}_m \|_{L^1(\Omega)} \le \varepsilon 2^{-2m}$ and $\int_\Omega \lip \,\tilde{u}_m\,\d \mm \le 2 |D u|(\Omega)$. Hence we can rewrite 
\[
\tilde{u} = \sum_{i=0}^m \varphi_i u_i = \sum_{i=0}^m \varphi_i (u_i-\tilde{u}_m)+ \tilde{u}_m\qquad \text{on }\bigcup_{i=0}^m \Omega_i.
\]
We estimate the slope of $\tilde{u}$
\begin{equation*}
    \lip\,\tilde{u} = \sum_{i=0}^m \left[\,\lip \,\varphi_i \,|u_i-\tilde{u}_m|+\varphi_i\,(\lip\,u_i+\lip\,\tilde{u}_m) \right]+ \lip\,\tilde{u}_m\qquad\text{on }\bigcup_{i=0}^m \Omega_i
\end{equation*}
and integrate 
\begin{equation*}
\begin{aligned}
\int_{\cup_{i=0}^m \Omega_i} \lip\,\tilde{u}\,\d \mm & \le \sum_{i=0}^m \int_{\Omega_i} \left[\,\lip \,\varphi_i \,|u_i-\tilde{u}_m|+\varphi_i\,(\lip\,u_i+\lip\,\tilde{u}_m) \right]\,\d \mm+ \int_{\Omega} \lip\,\tilde{u}_m\,\d \mm\\
& \le \sum_{i=0}^m \left[ \varepsilon 2^i (2^{-2i}+2^{-2m}) + 2 |Du|(\Omega_i) \right] + \int_{\cup_{i=0}^m \Omega_i} \sum_{i=0}^m \varphi_i\,\lip\, \tilde{u}_m\,\d \mm \\
&\quad + 2 |D u|(\Omega) \\
& \stackrel{\eqref{eq:sum_tv_strips}}{\le} 4\varepsilon + 8 |D u|(\Omega).
\end{aligned}
\end{equation*}
\end{proof}
\subsection{Main propositions}
\label{sec:main_propositions}
In this section we conclude the proofs of Theorem \ref{thm:boundaryzero} and Theorem \ref{thm:general} by proving the implications between $W^{1,1}$-, $W_w^{1,1}$- and strong $BV$-extensions. The connection between strong perimeter extension and strong $BV$-extension was already shown in Proposition \ref{prop:strong_equivalence}.

The smoothing argument is a key tool in the proof of the following chain of implications.
\begin{figure}[h]
\[\begin{tikzcd}
	{W^{1,1}} &&&&&&&&&& {W_w^{1,1}} \\
	\\ \\ \\
	&&&&& {\text{s-}BV}
	\arrow["{\text{(Prop.\ \ref{prop:from_wW11_to_sBV})}\,\text{if }\mathfrak{m}(\partial \Omega) = 0}"', from=1-11, to=5-6]
	\arrow["{\text{(Prop.\ref{prop:from_W11_to_wW11})}}", from=1-1, to=1-11]
	\arrow["{(\text{Prop.\ \ref{prop:sBV_to_W11})}}"', from=5-6, to=1-1]
	\arrow["{\text{false if }\mathfrak{m}(\partial \Omega) >0\,(\text{Example\ \ref{ex:slit}})}"{description}, curve={height=24pt}, dashed, from=1-11, to=1-1]
    \arrow["{\text{false if }\mathfrak{m}(\partial \Omega) >0\,(\text{Example\ \ref{ex:W11nonsBV}})}"{description}, curve={height=24pt}, dashed, from=1-1, to=5-6]
\end{tikzcd}\]
\centering
\caption{Summary of main propositions and examples of the section.}
\label{fig:diagram}
\end{figure}
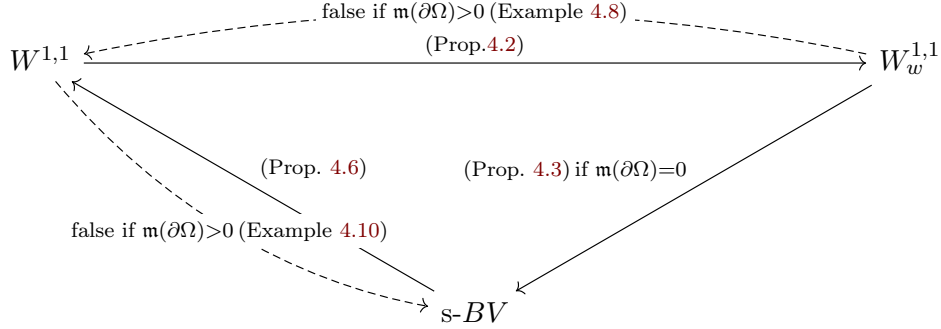

\begin{proposition}
\label{prop:from_W11_to_wW11}
Let $\Omega\subset\X$ be an open set. If $\Omega$ is a $W^{1,1}$-extension set, then $\Omega$ is a $W_w^{1,1}$-extension set.
\end{proposition}
\begin{proof}
We call the extension operator given by the assumption $E \colon W^{1,1}(\Omega) \to W^{1,1}(\X)$. Let $u \in W_w^{1,1}(\Omega)$. Since $W_w^{1,1}(\Omega) \subset BV(\Omega)$, we are in a position to apply the smoothing operator $T_\varepsilon \colon BV(\Omega) \to W^{1,1}(\Omega)$ from Proposition \ref{prop:smoothing} to $u$, defining
\begin{equation}
    \widetilde{E} u:= \begin{cases}
    u & \text{ in }\Omega,\\
    E T_{\varepsilon} u&\text{ in }\X \setminus \Omega.\\
    \end{cases}
\end{equation}
By Proposition \ref{prop:smoothing} we then have $|D(u-T_\varepsilon u)\chi_{\Omega}|(\partial\Omega) = 0$.
Hence, for every open set $U \subset \X$, we have
\begin{equation*}
\begin{aligned}
    |D((u-T_\varepsilon u)\chi_{\Omega})|(U)&=|D(u-T_\varepsilon u)|(\Omega \cap U)+|D(u-T_\varepsilon u)|(\partial\Omega \cap U)\\
    &\quad + |D((u-T_\varepsilon u)\chi_{\Omega})|((\X \setminus \bar{\Omega}) \cap U) = |D(u-T_\varepsilon u)|(\Omega \cap U) \\
    & \le |Du|(\Omega\cap U)+|D T_\varepsilon u|( \Omega\cap U) \le \int_{\Omega \cap U} \frac{\d|D u|}{\d\mm} + \lip \,T_\varepsilon u\,\d \mm.
\end{aligned}
\end{equation*}
This gives that, for every $F \in \mathscr{B}(\X)$,
\begin{equation}
\label{eq:quant_abs_cont}
|D((u-T_\varepsilon u)\chi_{\Omega})|(F)\le \int_{\Omega \cap F} \frac{\d|D u|}{\d\mm} + \lip \,T_\varepsilon u\,\d \mm.
\end{equation}
Therefore, if $\mm(F) =0$, then $|D((u-T_\varepsilon u)\chi_{\Omega})|(F) = 0$. Thus $((u-T_\varepsilon u)\chi_{\Omega}) \in W_w^{1,1}(\X)$.
Moreover, by the definition of $E$, $E T_\varepsilon u \in W^{1,1}(\X) \subset W_w^{1,1}(\X)$; so $\widetilde{E} u = (u-T_\varepsilon u) \chi_{\Omega} + E T_\varepsilon u \in W_w^{1,1}(\X)$. 
Notice that it follows by the very definition of total variation that
\begin{equation}
    \label{eq:tot_var_bdd_by_intslope}
    |D T_\varepsilon u|(\Omega) \le \int_{\Omega} \lip T_\varepsilon u\,\d \mm.
\end{equation}
We estimate
\begin{equation}
\label{eq:L1norm_fW11towW11}
    \| \widetilde{E} u \|_{L^1(\X)} \le \| u \|_{L^1(\Omega)}+ \| E T_\varepsilon u \|_{L^1(\X)}
\end{equation}
and
\begin{equation}
\label{eq:seminorm_fW11towW11}
\begin{aligned}
    \left\|\frac{\d |D \widetilde{E} u|}{\d \mm} \right\|_{L^1(\X)} &= |D \widetilde{E} u|(\X) \le |D((u-T_\varepsilon u)\chi_{\Omega})|(\X)+ |D E T_\varepsilon u|(\X)\\
    &\stackrel{\eqref{eq:quant_abs_cont}}{\le} \left\| \frac{\d |D u|}{\d \mm} \right\|_{L^1(\Omega)}+ \| \lip \,T_\varepsilon u \|_{L^1(\Omega)}+ |D E T_\varepsilon u| (\X).\\
\end{aligned}
\end{equation}
Hence, summing up \eqref{eq:L1norm_fW11towW11} and \eqref{eq:seminorm_fW11towW11}, we get
\begin{equation*}
\begin{aligned}
    \| \widetilde{E} u \|_{W_w^{1,1}(\X)} & \le \| u\|_{W^{1,1}_w(\Omega)} + \| \lip \,T_\varepsilon u \|_{L^1(\Omega)}+\| E\|_{W^{1,1}} \|T_\varepsilon u \|_{W^{1,1}(\Omega)}\\ 
    &\stackrel{\eqref{eq:smoothing_normbounds}}{\le} \| u \|_{W_w^{1,1}(\Omega)}+ C(|D u|(\Omega)+\varepsilon)+ \| T_\varepsilon u \|_{W^{1,1}(\Omega)} \stackrel{\eqref{eq:smoothing_normbounds},\eqref{eq:tot_var_bdd_by_intslope}}{\le} C(\| u \|_{W_w^{1,1}(\Omega)}+\varepsilon),
\end{aligned}
\end{equation*}
where $C=C(\| E\|_{W^{1,1}})$, thus concluding when choosing $\varepsilon:=\| u\|_{W_w^{1,1}(\Omega)}$.
\end{proof}
\begin{proposition}
\label{prop:from_wW11_to_sBV}
Let $\Omega$ be an open set such that $\mm(\partial \Omega) =0$. If $\Omega$ is a $W_w^{1,1}$-extension set, then $\Omega$ is also a strong $BV$-extension set.
\end{proposition}
\begin{proof}
We call the extension operator given by the assumption $E \colon W_w^{1,1}(\Omega) \to W_w^{1,1}(\X)$. Let $u \in BV(\Omega)$. We are in a position to apply the smoothing $T_\varepsilon \colon BV(\Omega) \to W^{1,1}(\Omega)$ to $u$, defining
\begin{equation}
    \widetilde{E} u:= \begin{cases}
    u & \text{ in }\Omega,\\
    E T_{\varepsilon} u&\text{ in }\X \setminus \Omega.\\
    \end{cases}
\end{equation}
 We check that $\widetilde{E} u \in BV(\X)$; We rewrite $\widetilde{E} u$ as follows
 \begin{equation}
    \widetilde{E}u= (u-T_\varepsilon u) \chi_{\Omega} + E T_\varepsilon u.
 \end{equation}
 From Proposition \ref{prop:smoothing}, $(u-T_\varepsilon u)\,\chi_{\Omega} \in BV(\X)$ and $|D((u-T_\varepsilon u)\,\chi_{\Omega})|(\partial \Omega) =0$; moreover, $\widetilde{E}T_\varepsilon u \in W_w^{1,1}(\X)\subset BV(\X)$, so also $\widetilde{E} u \in BV(\X)$. We check that $|D \widetilde{E}u|(\partial \Omega) =0$. Indeed, using that $|D((u-T_\varepsilon u)\,\chi_{\Omega})|(\partial \Omega) =0$,
\begin{equation}
\label{eq:N11_to_BV_bdry}
|D \widetilde{E} u|(\partial \Omega) \le |D E T_\varepsilon u|(\partial \Omega) = \int_{\partial \Omega} \frac{\d |D E T_\varepsilon u|}{\d \mm}\,\d \mm =0,
\end{equation}
where the last equality follows from the fact that $\partial \Omega$ is $\mm$-negligible. We have that $\| \widetilde{E} u \|_{L^1(\X)} = \| E T_\varepsilon u \|_{L^1(\X \setminus \bar{\Omega})}+ \| u \|_{L^1(\Omega)}$ and
\begin{equation}
    \begin{aligned}
    |D\widetilde{E}u|(X) & \le |D(u-T_\varepsilon u)\chi_{\Omega}|(\X)+|D({E}(T_\varepsilon u))|(\X) \\
    & = |D(u-T_\varepsilon u)\chi_{\Omega}|(\Omega)+ |D(u-T_\varepsilon u)\chi_{\Omega}|(\partial \Omega)\\
    & \quad + |D(u-T_\varepsilon u)\chi_{\Omega}|(\X\setminus \bar{\Omega})+|D(E(T_\varepsilon u))|(\X) \\
    &\le |Du|(\Omega) + |D T_\varepsilon u|(\Omega) +|D(E(T_{\varepsilon}u))|(X).
    \end{aligned}
\end{equation}
Therefore, we have
\begin{equation}
\begin{aligned}
    \| \widetilde{E} u\|_{BV(\X)} &\le \| u \|_{BV(\Omega)}+|D T_\varepsilon u|(\Omega)+ \| E(T_\varepsilon u) \|_{BV(\X)}\\
    &\le \| u\|_{BV(\Omega)}+|D T_\varepsilon u|(\Omega)+\| E\|_{W^{1,1}_w}\| T_\varepsilon u \|_{W^{1,1}_w(\Omega)}\\
    & \stackrel{\eqref{eq:bound_tot_variation_fLip}}{\le} C\left(\| u\|_{BV(\Omega)}+\| T_{\varepsilon} u \|_{L^1(\Omega)}+\int_{\Omega} \,\lip T_\varepsilon u\,\d \mm\right) \le C(\| u \|_{BV(\Omega)}+\varepsilon),\\
\end{aligned}    
\end{equation}
where $C=C(\| E\|_{W^{1,1}_w})$, concluding the proof with the choice $\varepsilon:= \| u\|_{BV(\Omega)}$.
\end{proof}
The following lemma is needed for the implication that strong $BV$-extension sets are $W^{1,1}$-extension sets.
\begin{lemma}
\label{lemma:fromlocal_to_global_W11}
Let $f \in BV(\X)$ such that $f = u$ in $\Omega$ with $u \in W^{1,1}(\Omega)$ and $f = v$ on $\X \setminus \bar{\Omega}$ with $v \in W^{1,1}(\X\setminus \bar{\Omega})$. Moreover, we assume that $|D f|(\partial \Omega) = 0$. Then $f \in W^{1,1}(\X)$ and
\begin{equation}
\label{eq:estimate_1mwug_gluing}
|D f|_{1,\X} \le \chi_{\Omega} |D u|_{1,\Omega} + \chi_{\X \setminus \bar{\Omega}} |D v|_{1,\X\setminus \bar{\Omega}} \quad \mm\text{-a.e.}
\end{equation}
\end{lemma}
\begin{proof}
Fix an $\infty$-test plan $\ppi$. Since $\chi_{\Omega} |D u| + \chi_{\X \setminus \bar{\Omega}} |D v| \in L^1(\mm)$, it is enough to show that, for $\ppi$-a.e. $\gamma$, 
\begin{equation}
\label{eq:gluing_estimate_claim}
|f(\gamma_1)- f(\gamma_0)| \le \int_0^1 (\chi_{\Omega} |D u|_{1,\Omega} + \chi_{\X \setminus \bar{\Omega}} |D v|_{1,\X\setminus \bar{\Omega}})(\gamma_t)|\dot{\gamma}_t|\,\d t.
\end{equation}
It follows by Theorem \ref{thm:BVequiv} that, for every $\gamma \in \Gamma(\X) \setminus N_0$, where $\ppi(N_0) = 0$, $f \circ \gamma \in BV(0,1)$, $|f(\gamma_1)-f(\gamma_0)| \le |D(f \circ \gamma)|(0,1)$ and $|D (f \circ \gamma)|(\gamma^{-1}(\partial \Omega)) = 0$.\\
We consider the sets: $A_1:= \Omega$, $A_2:=\X \setminus \bar{\Omega}$.\\
Given $t,s \in [0,1]$ and $A \subset \X$, we define $C_{t,s}^{A}:=\{ \gamma: \gamma([t,s]) \subset A\}$.
Notice that, if $A$ is open, $C_{t,s}^A = \cup_{k} C_{t,s}^{B_k}$, where $B_k$ is an increasing sequence of closed sets such that $A = \cup_{k} B_k$. Since $C_{t,s}^{B_k} \subset \Gamma(\X)$ is closed, we have that $C_{t,s}^A$ is Borel.
For every $q_1,q_2 \in \mathbb{Q} \cap [0,1]$, $i \in \{1,2\}$, we will consider $C_{q_1,q_2}^{A_i} =: C_{q_1,q_2}^{i}$, which are Borel sets.
Therefore, we consider, if $\ppi(C_{q_1,q_2}^{i}) >0$,
\[ \ppi_{q_1,q_2}^i:= \ppi(C_{q_1,q_2}^i)^{-1}\,\ppi \restr{C_{q_1,q_2}^i}\text{ and }\tilde{\ppi}_{q_1,q_2}^i:= ({{\rm restr}_{q_1,q_2}})_*\ppi_{q_1,q_2}^i.\]
It can be readily checked that $\tilde{\ppi}_{q_1,q_2}^i$ is a test plan on $A_i$.\\
We consider the case $i =1$. We can find a $\tilde{\ppi}_{q_1,q_2}^i$-negligible set $\tilde{N}_{q_1,q_2}^1$ such that for every $\gamma \in \Gamma(\X) \setminus \tilde{N}_{q_1,q_2}^1$
\[ |u(\gamma_1)-u(\gamma_0)| \le \int_{\gamma} |D u|_{1,\Omega} .\]
We define $N_{q_1,q_2}^1:= {\rm restr}_{q_1,q_2}^{-1}(\tilde{N}_{q_1,q_2}^1)$ which is $\ppi_{q_1,q_2}^1$- negligible, so it is $\ppi$-negligible and we can assume without loss of generality that $N_{q_1,q_2}^1 \subset C_{q_1,q_2}^1$
and for every $\gamma \in C_{q_1,q_2}^1 \setminus N_{q_1,q_2}^1$
\[ |u(\gamma_{q_2}) -u(\gamma_{q_1})| \le \int_{q_1}^{q_2} |D u|_{1,\Omega} (\gamma_t)|\dot{\gamma}_t|\,\d t.\]
We can argue similarly for the case $i = 2$, defining $\ppi$-negligible sets $N_{q_1,q_2}^2 \subset C_{q_1,q_2}^2$ being such that for every $\gamma \in C_{q_1,q_2}^2 \setminus  N_{q_1,q_2}^2$
\[ |v(\gamma_{q_2})-v(\gamma_{q_1})| \le \int_{q_1}^{q_2} |D v|_{1,\X\setminus \bar{\Omega}}(\gamma_t)|\dot{\gamma}_t|\,\d t.\]
We define the set $N = N_0 \cup \cup_{q_1,q_2}(N_{q_1,q_2}^1 \cup N_{q_1,q_2}^3)$, which is $\ppi$-negligible and we claim that for every $\gamma \in \Gamma(\X) \setminus N$ \eqref{eq:gluing_estimate_claim} holds. 
Denoting $U_n := B_{2^{-n}}(\partial \Omega)$, we have that, for every $\gamma \in \Gamma(\X) \setminus N$,
\begin{equation}
\label{eq:estimate_tubular_neigh}
    \lims_{n\to \infty} |D (f \circ \gamma)|(\gamma^{-1}(U_n)) =0.
\end{equation}
We fix $n$ and consider $A_3:=U_n$. We consider the set
\[ \mathscr{I}:= \{ I =(a,b) \cap [0,1]:\,a,b \in \mathbb{Q},\quad\gamma([a,b])\subset A_i\quad \text{for some }i \}.\]
Notice that $\mathscr{I}$ is a countable cover of open (in the induced topology of $[0,1]$) sets of $[0,1]$; therefore, there exists a finite subcover $[0,b_0),(a_i,b_i)_{i=1}^N,(a_{N+1},1)$.
For every $1 \le i \le N$, choose $c_i \in \mathbb{Q}$ such that $a_i < c_i < b_{i-1}$ and define $c_0 = 0$, $c_{N+1} = 1$; therefore, we estimate
\begin{equation}
\begin{aligned}
    |f(\gamma_1)-f(\gamma_0)| &\le \sum_{i=0}^N |f(\gamma_{c_{i+1}})-f(\gamma_{c_i})| \le \sum_{j=1}^3 \sum_{i: \gamma([c_i,c_{i+1}]) \subset A_j} |f(\gamma_{c_{i+1}})-f(\gamma_{c_i})|\\
    & \le \int_{\gamma \cap \Omega} |D u|_{1,\Omega} + \int_{\gamma \cap (\X \setminus \bar{\Omega})} |D v|_{1,\X\setminus \bar{\Omega}}+ \sum_{i: \gamma([c_i,c_{i+1}])\subset A_3} |f(\gamma_{c_{i+1}}) -f(\gamma_{c_i})|\\
    & \le \int_{\gamma \cap \Omega} |D u|_{1,\Omega} + \int_{\gamma \cap (\X \setminus \bar{\Omega})} |D v|_{1,\X\setminus \bar{\Omega}}+ \sum_{i: \gamma([c_i,c_{i+1}])\subset A_3} |D(f \circ \gamma)|(c_i,c_{i+1})\\
    &\le \int_{\gamma \cap \Omega} |D u|_{1,\Omega} + \int_{\gamma \cap (\X \setminus \bar{\Omega})} |D v|_{1,\X\setminus \bar{\Omega}}+ |D (f \circ \gamma)|(\gamma^{-1}(U_n)).
\end{aligned}
\end{equation}
By taking the $\lims$ as $n \to \infty$, using \eqref{eq:estimate_tubular_neigh}, we obtain \eqref{eq:gluing_estimate_claim}, thus proving the claim and concluding the proof.
\end{proof}
\begin{remark}
We point out that, under the hyphothesis of Lemma \ref{lemma:fromlocal_to_global_W11} and the notation therein, we have that 
\begin{equation}
\label{eq:equality_1mwug_on_Omega_and_X}
|D f|_{1,\X} = |D f|_{1,\Omega}\quad\mm\text{-a.e. on }\Omega.
\end{equation}
Indeed, the inequality $\ge$ follows from \eqref{eq:1mwug_on_Omega}, while the converse one from \eqref{eq:estimate_1mwug_gluing}.
\end{remark}
\begin{proposition}
\label{prop:sBV_to_W11}
Let $\Omega\subset \X$ be open. If $\Omega$ is a strong $BV$-extension set, then it is a $W^{1,1}$-extension set.
\end{proposition}
\begin{proof}
Consider $u \in W^{1,1}(\Omega)$ and its minimal 1-weak upper gradient $|D u|_{1,\Omega}$. Since $W^{1,1}(\Omega)\subset BV(\Omega)$, by assumption we have the existence of a strong $BV$-extension operator $F \colon BV(\Omega) \to BV(\X)$. Then we define
\[ \tilde{F} u := \begin{cases} 
F u &\text{ on }\bar{\Omega},\\
T_\varepsilon (F u \restr{\X \setminus \bar{\Omega}}) & \text{ on }\X \setminus \bar{\Omega},\\
\end{cases}\]
where $T_\varepsilon \colon BV(\X\setminus \bar{\Omega}) \to {\rm Lip}_{\textrm{loc}}(\X\setminus \bar{\Omega})$ is the smoothing operator of Proposition \ref{prop:smoothing}, when applied to the open set $\X \setminus \bar{\Omega}$ with $\varepsilon$ to be chosen later.
We rewrite the function $\tilde{F} u = F u + (T_\varepsilon (F u\restr{\X \setminus \bar{\Omega}})-F u\restr{\X \setminus \bar{\Omega}})\chi_{\X \setminus \bar{\Omega}} \in BV(\X)$ as it is sum of $BV$ functions on $\X$.
Moreover, by Proposition \ref{prop:smoothing} and the fact that $F$ is a strong $BV$-extension operator 
\begin{equation}
\label{eq:tot_variation_tildefu}
    |D \tilde{F} u|(\partial \Omega) = 0.
\end{equation}
By definition $\tilde{F} u  = u$ $\mm$-a.e.\ on $\Omega$. We check that $\tilde{F} u \in W^{1,1}(\X)$. Indeed, since $u \in W^{1,1}(\Omega)$, $T_\varepsilon (F u \restr{\X \setminus \bar{\Omega}}) \in W^{1,1}(\X \setminus \bar{\Omega})$ and \eqref{eq:tot_variation_tildefu} holds, we are in position to apply Lemma \ref{lemma:fromlocal_to_global_W11}, thus having that $\tilde{F} u \in W^{1,1}(\X)$ and
\[ |D \tilde{F} u|_{1,\X} \le |D u|_{1,\Omega} \,\chi_{\Omega}+ | D(T_\varepsilon (F u \restr{\X \setminus \bar{\Omega}}))|_{1,\X\setminus \bar\Omega}\,\chi_{\X \setminus \bar{\Omega}} \le |D u|_{1,\Omega} \,\chi_{\Omega}+  \lip(T_\varepsilon (F u \restr{\X \setminus \bar{\Omega}}))\chi_{\X \setminus \bar{\Omega}}.\]
Thus, integrating, we get that 
\begin{equation}
\label{eq:estimate_seminorm_W11_ext}
\begin{aligned}
\| |D \tilde{F} u | \|_{L^1(\mm)} & \le \||D u|\|_{L^1(\Omega)}+ \| \lip(T_\varepsilon (F u \restr{\X \setminus \bar{\Omega}})\|_{L^1(\X \setminus \bar{\Omega})}\\
& \stackrel{\eqref{eq:smoothing_normbounds}}{\le} \||D u|\|_{L^1(\Omega)}+ C(|D F u|(\X \setminus \bar{\Omega}) +\varepsilon)\\
& \le \| |D u|\|_{L^1(\Omega)}+ C(|D F u|(\X) +\varepsilon)  \le  \| |D u|\|_{L^1(\Omega)}+ C(\| u\|_{BV(\Omega)} +\varepsilon).\\
\end{aligned}
\end{equation}
To conclude, we compute
\begin{equation}
\label{eq:estimate_w11loc_to_bvloc}
    |D u|(\Omega)=|D \tilde{F} u|(\Omega) \le \| |D \tilde{F} u|_{1,\X} \|_{L^1(\Omega)} \stackrel{\eqref{eq:equality_1mwug_on_Omega_and_X}}{=} \| |D  u|_{1,\Omega} \|_{L^1(\Omega)},
\end{equation}
thus we can continue the estimate in \eqref{eq:estimate_seminorm_W11_ext}, having that $\| |D \tilde{F} u | \|_{L^1(\mm)} \le C(\| u\|_{W^{1,1}(\Omega)} +\varepsilon)$.
Then
\begin{equation}
    \| \tilde{F} u \|_{L^1(\X)} = \| F u \|_{L^1(\bar{\X})}+\varepsilon \le C(\| 
 u\|_{L^1(\Omega)} + |D u|(\Omega) +\varepsilon) \le C(\| 
 u\|_{W^{1,1}(\Omega)} +\varepsilon)
\end{equation}
where the first inequality follows from the very definition of $\tilde{F} u$ and \eqref{eq:smoothing_normbounds}, the second one from the fact that $F$ is a $BV$-extension operator and the last one from \eqref{eq:estimate_w11loc_to_bvloc}.
By choosing $\varepsilon = \| u\|_{W^{1,1}(\Omega)}$, we conclude.
\end{proof}
\subsection{Examples}
In this last subsection, we provide several examples of a metric measure spaces $(\X,\sfd,\mm)$ with $\mm(\partial \Omega) >0$ and open sets $\Omega \subset \X$ 
having some of the extension properties, but not others.
We start with a basic example from  \cite[Example 7.4]{Ambrosio-DiMarino14} showing that $W^{1,1}(\X)$ is not the same as $W_w^{1,1}(\X)$.
\begin{example}\label{ex:disk}
 Let our space $\X$ be $\mathbb R^2$ equipped with the Euclidean distance and the measure $\mm = \mathcal L^2 + \mathcal H^1|_{\partial B(0,1)}$,
Then the domain $\Omega = B(0,1) \subset \mathbb R^2$ does not have the strong $BV$, the $W^{1,1}$, nor the $W_w^{1,1}$ extension property.
\end{example}

Let us then give an example which is a $W_w^{1,1}$-extension domain, but not a $W^{1,1}$-extension domain.

\begin{example}
\label{ex:slit}
We consider the metric measure space $(\X,\sfd,\mm)=(\mathbb{R}^2,\sfd_e,\mm)$, where $\sfd_e$ is the $2$- Euclidean distance and $\mm:=\mathcal{L}^2+\mathscr{H}^1 \restr{S}$ and $S:=[0,1] \times \{ 0\}$. We consider the open set $\Omega:=\mathbb{R}^2
\setminus S$. In particular, we notice that $\mm(\partial \Omega) = \mathscr{H}^1(S) >0$. We claim that $\Omega$ is a $W_{w}^{1,1}$-extension domain, but it is not a $W^{1,1}$-extension domain.\\
We consider a function $u$ such that $u = 1$ on $[\frac{1}{3},\frac{2}{3}]\times [0,\frac{1}{3}]$, supported on $[0,1] \times \mathbb{R}$ and Lipschitz on its support. Then it follows by the very definition of $W^{1,1}(\Omega)$ that $u \in W^{1,1}(\Omega)$. Assume by contradiction that $\Omega$ is a $W^{1,1}$-extension domain and consider the extension, say $v \in W^{1,1}(\X)$. We consider $\mu:= \mathcal{L}^2(B(z,1/6))^{-1}\, \mathcal{L}^2 \restr{B(z,1/6)}$ with $z:=(\frac{1}{2},\frac{1}{6})$ and we define the map $F(x,t):= x-\frac{t}{3}e_2$ for $t \in [0,1]$ and $x \in \mathbb{R}^2$. We denote by $G \colon \X \to \Gamma(\X)$ the map defined as $(G(x))_t:=F(x,t)$. 
We define $\ppi:= G_* \mu$ which can be readily checked to be an $\infty$-test plan. We have that, for $\ppi$-a.e. $\gamma$, $v \circ \gamma$ is not $W^{1,1}(0,1)$, hence contradicting item ii) in Proposition \ref{prop:equivalence_localized_W11} for any choice of $G \in L^1(\mm)$. Hence, $v \notin W^{1,1}(\X)$.
We now show that $\Omega$ is a $W_w^{1,1}$-extension domain.
We denote respectively by $W^{1,1}_e(\Omega)$ and $BV_e(\Omega)$ the $W^{1,1}$ and $BV$ spaces on $\Omega$ in the mms $(\mathbb{R}^2,|\cdot|,\mathcal{L}^2)$; moreover, we denote by $\nabla u$ the distributional gradient of $u$. Since $\mm\restr{\Omega}=\mathcal{L}^2\restr{\Omega}$, we have $W^{1,1}_w(\Omega) = W^{1,1}_e(\Omega)$ with
\begin{equation}
    \left\|\frac{\d|D v|}{\d \mm}\right\|_{L^1(\Omega)} = \| | \nabla v| \|_{L^1(\Omega,\mm)}.
\end{equation}
The goal here is to construct the $W^{1,1}_w$ extension operator. We consider $u \in W^{1,1}_w(\Omega)$.
We define $\Omega^+ := \{ y >0\}$ and $\Omega^-:= \{ y <0\}$. By the theory of traces, since $u \in W^{1,1}(\Omega^+)$, there exists $u^+ \in L^1(\partial \Omega^+)$ such that for every $\varphi \in C^\infty_c(\mathbb{R}^2,\mathbb{R}^2)$
\begin{equation}
\label{eq:integration_by_parts_U+}
    \int_{\Omega^+} \varphi \cdot \nabla u\,\d \mathcal{L}^2+\int_{\Omega^+} u\,\div \varphi\,\d \mathcal{L}^2 = \int_{\partial \Omega^+} \varphi \cdot \nu\,u^+\,\d \mathscr{H}^1.
\end{equation}
The same holds on $\Omega^-$ and it can be readily checked that $u^+ = u^-$ on $\{ y= 0 \}\setminus S$. Therefore, summing up \eqref{eq:integration_by_parts_U+} and the same term for $\Omega^-$, we get that, for every $\varphi \in C^\infty_c(\mathbb{R}^2,\mathbb{R}^2)$
\begin{equation*}
    \int \varphi \cdot \nabla u\,\d \mathcal{L}^2+\int u\,\div \varphi\,\d \mathcal{L}^2 = \int_{S} \varphi \cdot e_2\,(u^+-u^-)\,\d \mathscr{H}^1.
\end{equation*}
Hence, $u \in BV_e(\mathbb{R}^2)$ and $Du = \nabla u \,\mathcal{L}^2+ (u^+-u^-)e_2 \,\mathscr{H}^1\restr{S}$. We define $\tilde{u} \in L^1(\mm)$ to be equal to $u$ $\mathcal{L}^2$-a.e.\ on $\Omega$ and to $0$ $\mathscr{H}^1$-a.e.\ on $S$. We claim that $\tilde{u} \in BV(\mathbb{R}^2,\sfd_e,\mm)$ and 
\begin{equation}
    |D \tilde{u}|_{\mm} \le |\nabla u| \mathcal{L}^2\restr{\mathbb{R}^2\setminus S} + (|u^+|+|u^-|-|u^+-u^-|)\mathscr{H}^1\restr{S}=:\nu.
\end{equation}
Since $u \in W^{1,1}_e(\Omega)$, we know that there exists $u_k\in W^{1,1}_e(\Omega)\cap C^{\infty}(\Omega)$ such that $u_k \to u$ in $W^{1,1}(\Omega)$. Up to passing to a strongly convergent subsequence, we can assume that there exists $H \in L^1(\Omega,\mm)$ such that $|\nabla u_k| \le H$ $\mm$-a.e.\ on $\Omega$ for every $k$.
Moreover, for every $k$, we define
\begin{equation}
\varphi_k := (k\,\sfd(B(S,\frac{1}{k}),\cdot))\wedge 1
\end{equation}
Notice that $\varphi_k u_k=0$ $\mathscr{H}^1$-a.e.\ on $S$ and that $\lip\,\varphi_k= k \chi_{B(S,\frac{2}{k})\setminus B(S,\frac{1}{k})}$ $\mm$-a.e..
Moreover, $\varphi_k u_k \to \tilde{u}$ in $L^1(\Omega,\mm)$; so, $\varphi_k u_k$ is an admissible competitor in the definition of $|D\tilde{u}|_{\mm}$ on open sets.
We consider an open cube $Q\subset \mathbb{R}^2$.
If $Q \cap S = \emptyset$, we have that $|D \tilde{u}|_\mm(Q) =\nu(Q)$.
If $Q \cap S \neq \emptyset$, we do the following. By the Leibniz formula for $\lip$ and the fact that $u_k \in C^{\infty}(\Omega)$, we get
\begin{equation}
    \int_Q \lip (\varphi_k u_k)\,\d\mm \le \int_Q \varphi_k\,|\nabla u_k|\,\d\mm + k
    \int_{Q \cap B(S,\frac{2}{k})\setminus B(S,\frac{1}{k}) }\, u_k\,\d\mm = (A_k) + (B_k)
\end{equation}
Firstly, we estimate the term $(A_k)$:
\begin{equation}
    \lims_{k\to \infty} \int_Q \varphi_k\,|\nabla u_k|\,\d\mm \le \lim_{k\to \infty} \int_Q |\nabla u_k|\,\d\mm = \| |\nabla u| \|_{L^1(Q)}. 
\end{equation}
Secondly, we estimate $(B_k)$:
\begin{equation}
    (B_k) \le \int_{S\cap Q} \avint_{\frac{1}{k}}^{\frac{2}{k}} |u_k|\,\d y\,\d x + \int_{S\cap Q} \avint_{-\frac{2}{k}}^{-\frac{1}{k}} |u_k|\,\d y\,\d x +o(1).
\end{equation}
We estimate the first term in the last equation as follows:
\begin{equation}
\begin{aligned}
    \lims_{k \to \infty} \int_{S\cap Q} \avint_{\frac{1}{k}}^{\frac{2}{k}} |u_k|\,\d y\,\d x &\le \lims_{k\to \infty} \int_{S \cap Q} \avint_{\frac{1}{k}}^{\frac{2}{k}} |u_k(x,y)|-|u_k(x,0)|\,\d y\,\d x + \lims_{k\to \infty} \int_{S \cap Q} |u_k(x,0)|\,\d x\\
    & \le \lims_{k\to \infty} \int_{S \cap Q} \avint_{\frac{1}{k}}^{\frac{2}{k}} |u_k(x,y)|-|u_k(x,0)|\,\d y\,\d x + \int_{S \cap Q} |u^+(x)|\,\d x\\
\end{aligned}
\end{equation}
where in the last inequality we used the continuity of the trace operator from $W^{1,1}_e(\Omega)$ to $L^1(\partial \Omega)$.
We continue estimating the first addendum in the last line, having
\begin{equation}
\begin{aligned}
    \lims_{k \to \infty} &\int_{S \cap Q} \avint_{\frac{1}{k}}^{\frac{2}{k}} |u_k(x,y)|-|u_k(x,0)|\,\d y\,\d x \le \lims_{k \to \infty} \int_{S \cap Q} \avint_{\frac{1}{k}}^{\frac{2}{k}} \int_0^y \partial_z |u_k(x,z)|\,\d z\,\d y\,\d x\\
     & \le \lims_{k \to \infty} \int_{[0,1]\times [0, \frac{2}{k}]} |\nabla u_k|\,\d \mathcal{L}^2 \le \lims_{k \to \infty} \int_{[0,1]\times [0, \frac{2}{k}]} H \,\d \mathcal{L}^2=0,
\end{aligned}
\end{equation}
where the last equality follows by an application of dominated convergence theorem.
The same holds for the second term. By taking the limit as $k\to \infty$, we have proven that $|D \tilde{u}|_\mm(Q) \le \nu(Q)$.
By an application of monotone class theorem, we get that $|D \tilde{u}|_\mm(B) \le \nu(B)$ for every Borel set $B$, thus proving the claim.
Hence we got that $\tilde{u} \in BV(\mathbb{R}^2,\sfd_e,\mm)$, $|D \tilde{u}|_{\mm} \ll \mm$, thus $\tilde{u} \in W^{1,1}_w(\mathbb{R}^2,\sfd_e,\mm)$. 
To conclude the proof, it is enough to estimate its norm
\begin{equation}
\begin{aligned}
    \left\|\frac{\d|D \tilde{u}|_{\mm}}{\d \mm} \right\|_{L^1(\mm)}&\le \|\nabla u\|_{L^1(\Omega)}+ \int_S (|u^+|+|u^-|-|u^+-u^-|)\,\d \mathscr{H}^1 \\
    &\le \| |\nabla u |\|_{L^1(\Omega)}+ 2 \int_S (|u^+|+|u^-|)\,\d \mathscr{H}^1 \le C\|u\|_{W^{1,1}_w(\Omega)},\\
\end{aligned}
\end{equation}
where in the last inequality we applied the continuity of the trace operator. For what concern the $L^1$-norm, we have that $\| \tilde{u} \|_{L^1(\mm)} = \| u \|_{L^1(\Omega)}$ by the very definition of $\tilde{u}$. Hence, $\| \tilde{u} \|_{W^{1,1}_w(\X)}\le C \| u \|_{W^{1,1}_w(\Omega)}$.
\end{example}

Let us end this paper with examples of sets having different reasons for having the $W^{1,1}$-extension property, but not the strong $BV$-extension property.

\begin{example}\label{ex:W11nonsBVcheat}
Let us consider a bounded domain $\Omega = (-2,2)\times (0,1) \cup (-2,-1) \times (-1,0] \cup (1,2) \times (-1,0]\subset \mathbb R^2$ in $(\mathbb R^2,\sfd, \mm)$ with three different versions of distance and reference measure:
\begin{enumerate}
    \item  $\mm = \mathcal H^1|_{\mathbb R \times \{0\}}$  and $\sfd = \sfd_{\textrm{Euc}}$.
    \item $\mm = \mathcal H^1|_{\mathbb R \times \{0\}} + \mathcal L^2$ and $\sfd((x_1,y_1),(x_2,y_2)) = |x_1-x_2| + \sqrt{|y_1-y_2|}$.
    \item $\mm = \mathcal H^1|_{\mathbb R \times \{0\}} + \sum_i 2^{-i}\delta_{q_i}$, with $\{q_i\,:\, i \in \mathbb N\} = \mathbb Q^2$ and $\sfd = \sfd_{\textrm{Euc}}$.
\end{enumerate}
Before continuing, let us note that if we would take $\mm = \mathcal H^1|_{\mathbb R \times \{0\}} + \mathcal L^2$ and $\sfd = \sfd_{\textrm{Euc}}$, then similarly to Example \ref{ex:disk}, $\Omega$ would not have the $W^{1,1}$-extension property.

In all of the versions the strong $BV$-extension property fails because $[-1,1]\times \{0\} \subset \partial \Omega$ and any extension $Eu$ of a function $u \in BV(\Omega)$ that is $0$ on $(-2,-1)\times \{0\}$ and $1$ on $(1,2)\times \{0\}$ must have $|DEu|([-1,1]\times \{0\}) \ge 1$.

Let us briefly see why the domain in all the three cases has the $W^{1,1}$-extension property. In (1) the analysis reduces to $\mathbb R$ since the measure lives only on $\mathbb R \times \{0\}$. Since $(-2,-1) \cup (1,2)$ has the $W^{1,1}$-extension property in the Euclidean line, we conclude that also $\Omega$ has the $W^{1,1}$-extension property in $\mathbb R^2$.
The version (1) of the construction is perhaps not satisfactory due to the fact that the measure lives only on the line, so the rest of the space is superfluous. The version (2) and (3) address this by forcing the support of the measure to be the whole space.

In (2) different horizontal lines are not connected to each other by rectifiable curves, so the analysis again reduces to $\mathbb R$ as in (1). In (3) every point is connected by rectifiable curves, but the reference measure outside the line supporting the 1-dimensional Hausdorff measure does not support any non-trivial Sobolev structure. Hence, again the analysis reduces to $\mathbb R$ as in (1).
\end{example}

Even the last two versions (2) and (3) of Example \ref{ex:W11nonsBVcheat} are perhaps not so satisfactory because the relevant Sobolev-structure in them is restricted to horizontal directions.

In the last example of an open set with the $W^{1,1}$-extension property, but not the strong $BV$-extension property, the Sobolev-structure is richer, but consequently the construction and the verification of the extension properties is a bit more complicated.
\begin{example}\label{ex:W11nonsBV}
We will construct an open bounded set $\Omega \subset \mathbb R^2$ and a density $\rho \in L^1_{\textrm{loc}}(\mathbb R^2)$ so that $\Omega$ has the $W^{1,1}$-extension property in $(\mathbb R^2, \sfd_{\textrm{Euc}}, \rho\mathcal L^2)$, but it does not have the strong $BV$-extension property.

The open set $\Omega$ and the density $\rho$ are constructed using a sequence of balls $B_i = B(x_i,r_i)$
and a sequence of densities $\rho_i$ supported in $2B_i$. We start by enumerating $\{q_j\}_{j=1}^\infty = \mathbb Q \cap [-2,2]\times [-1,1]$ and define $r_1 = r_2 = 1$, $x_1=(-2,0)$, $x_2=(2,0)$, and $w_1(x) =w_2(x) = 1$ for all $x$. The remaining $x_i$, $r_i$, and $w_i$ will be defined by induction as follows. Suppose that $(x_i,r_i,w_i)_{i=1}^k$ have been defined. Let $j \in \mathbb N$ be the smallest integer so that $q_j \notin \bigcup_{i=1}^k \overline{B_i}$. Define
\[
 x_{k+1} = q_j, \qquad r_{k+1} = \min\left(2^{-(k+1)},\frac12\mathop{\sfd}(q_j,\bigcup_{i=1}^k B_i)\right),
\]
and
\[
\rho_{k+1}(x) = \begin{cases}\prod_{i=1}^{k+1}r_{i}^2, & \textrm{if }x \in B(x_{k+1},2r_{k+1}) \setminus B(x_{k+1},r_{k+1}),\\
1, & \textrm{otherwise}.
\end{cases}
\]
Finally, define 
\[
\Omega = \bigcup_{i=1}^\infty B_i \qquad \textrm{and} \qquad \rho(x) = \inf_i \rho_i(x).
\]

The reason why $\Omega$ is a $W^{1,1}$-extension set is that the small densities at the different annuli allow us to make cut-offs inside the annuli. Before justifying this, let us see why $\Omega$ does not have to strong $BV$-extension property. 
By the definition of $r_i$, we have
\[
\sum_{i=3}^\infty 4r_i \le \sum_{i=3}^\infty 2^{2-i} = 1.
\]
Therefore, the Lebesgue measure of
\[
A = \left\{y \in (-1,1)\,:\, ([-2,2] \times \{y\}) \cap \bigcup_{i=3}^\infty 2B_i = \emptyset \right\}
\]
is at least one. Along the line-segments $[-2,2] \times \{y\}$, with $y \in A$, the density $w$ is identically one. Consequently, if we consider the function $f \in BV(\Omega)$ that is identically one on $B_1$ and zero elsewhere, any $BV$-extension $Ef$ of it will satisfy
\[
|D(Ef)|\left(D\right) \ge 1,
\]
where
\[
D = \bigcup_{y\in A} ([-2,2] \times \{y\}) \setminus (B_1 \cup B_2).
\]
Notice that the definition of $\Omega$ forces $[-2,2]\times [-1,1] \subset \overline{\Omega}$, and so
$D \subset \partial\Omega$. Therefore, $\Omega$ does not have the strong $BV$-extension property.

Let us next check that $\Omega$ has the $W^{1,1}$-extension property. First we note that there exists a bounded extension operator $E \colon W^{1,1}(B(0,1)) \to W^{1,1}(\mathbb R^2)$ with respect to the Lebesgue measure.
We also take a cut-off function
$\varphi \in C_0^\infty(B(0,2))$ such that $\varphi = 1$ on $B(0,1)$. 
For every $i \in \mathbb N$ we define $T_i \colon \mathbb R^2 \to \mathbb R^2 \colon z \mapsto r_iz + x_i$.
The extension operator $E_\infty$ from $W^{1,1}(\Omega)$ to $W^{1,1}(\mathbb R^2)$ will be defined as a limit of extension operators
\[
E_k \colon W^{1,1}\left(\Omega_k\right) \to W^{1,1}(\mathbb R^2), \text{ with }\Omega_k = \bigcup_{i=1}^k B_i\text{ and reference measure }\inf_{i\le k} \mm_k = \rho_i(x)\mathcal L^2
\]
that are defined inductively as follows.  We define
\[
E_1u(x) = E(u\circ T_1|_{B(0,1)})(T_1^{-1}(x)).
\]
Supposing we have defined $E_k$, we set
\[
E_{k+1}u(x) = E_{k}u|_{\Omega_k}(x)(1-\varphi( T_{k+1}^{-1}(x))) + E(u\circ T_{k+1}|_{B(0,1)})(T_{k+1}^{-1}(x))\varphi(T_{k+1}^{-1}(x)).
\]
Since $E_{k}u|_{\Omega_k}(x) = E_{k-1}u|_{\Omega_{k-1}}(x)$ for all $x \notin \bigcup_{i=k}^\infty 2B_i$, and since $\mathcal L^2\left(\bigcup_{i=k}^\infty 2B_i\right) \to 0$ as $k \to \infty$, the definition
of the final extension operator as
\[
E_\infty u(x) = \lim_{k \to \infty} E_k u|_{\Omega_k}(x)
\]
is well posed.

Let us estimate the operator norm of $E_{k+1}$.
First of all, we have
\begin{align*}
 \int_{\mathbb R^2}(&|E_{k+1}u|+ |DE_{k+1}u|)\,\d\mm_{k+1}\\
 & \le \int_{ T_{k+1}(\mathbb R^2\setminus B(0,2))}+ \int_{T_{k+1}( B(0,1)) } + \int_{T_{k+1}(B(0,2)\setminus B(0,1)) }(|E_{k+1}u|+ |DE_{k+1}u|)\,\d\mm_{k+1}\\
 & \le \|E_k\|\|u\|_{W^{1,1}(\Omega_k)} + \|u\|_{W^{1,1}(B_i)}
 +\int_{T_{k+1}(B(0,2)\setminus B(0,1)) }(|E_{k+1}u|+ |DE_{k+1}u|)\,\d\mm_{k+1}.
\end{align*}
Let us estimate the last term. For the integral of the function we get, by the definitions of $E_{k+1}$ and of $\rho_i$
\[
 \int_{T_{k+1}(B(0,2)\setminus B(0,1)) }|E_{k+1}u|\,\d\mm_{k+1} \le r_{k+1}^2( \|E_k\|\|u\|_{W^{1,1}(\Omega_k)} + \|E\|\|u\|_{W^{1,1}(B_i)}).
\]
For the gradient part we first estimate the gradient via the product and chain rules for almost all $x \in T_{k+1}(B(0,2)\setminus B(0,1))$ by
\begin{align*}
 |DE_{k+1}u(x)| & \le |DE_{k}|_{\Omega_k}u(x)| + \frac{1}{r_{k+1}}|DE(u\circ T_{k+1}|_{B(0,1)})(T_{k+1}^{-1}(x))|\\
 & \quad + C\frac{1}{r_{k+1}}\left( |E_{k}u|_{\Omega_k}(x)| + |E(u\circ T_{k+1}|_{B(0,1)})(T_{k+1}^{-1}(x))|\right).
\end{align*}
Therefore,
\begin{align*}
 \int_{T_{k+1}(B(0,2)\setminus B(0,1)) }|DE_{k+1}u|\,\d\mm_{k+1}
 \le Cr_{k+1}( \|E_k\|\|u\|_{W^{1,1}(\Omega_k)} + \|E\|\|u\|_{W^{1,1}(B_i)}).
\end{align*}
We have thus obtained the estimate
\begin{align*}
\int_{\mathbb R^2}(|E_{k+1}u|+ |DE_{k+1}u|)\,\d\mm_{k+1}
& \le 
(1+Cr_{k+1})\|E_k\|\|u\|_{W^{1,1}(\Omega_k)} + (1+Cr_k\|E\|)\|u\|_{W^{1,1}(B_i)}\\
& \le (1 + C\|E\|r_k)\|E_k\|\|u\|_{W^{1,1}(\Omega_{k+1})}.
\end{align*}
Recalling that $r_k \le 2^{-k}$ for $k\ge 3$, by iterating the above, we have for all $k\ge 3$,
\begin{align*}
 \|E_{k+1}\| & \le (1 + C\|E\|2^{-k-1})\|E_k\| \\
 & \le \left(\prod_{i=3}^{k+1}(1 + C\|E\|2^{-i})\right)\|E_2\| \le \left(\prod_{i=3}^{\infty}(1 + C\|E\|2^{-i})\right)\|E_2\| < \infty.
\end{align*}
Since for all $u \in W^{1,1}(\Omega)$ we have
\begin{align*}
\|E_\infty u\|_{W^{1,1}(\mathbb R^2)} & = \lim_{k \to \infty} \|E_k u|_{\Omega_k}\|_{W^{1,1}(\mathbb R^2)} \le  \lim_{k \to \infty}\|E_{k}\|\|u|_{\Omega_k}\|_{W^{1,1}(\Omega_k)}\\
& \le \left(\prod_{i=3}^{\infty}(1 + C\|E\|2^{-i})\right)\|E_2\|\|u\|_{W^{1,1}(\Omega)},
\end{align*}
and so $E_\infty$ is bounded and we are done.
\end{example}

%
\bibliographystyle{abbrv}
\bibliography{biblio.bib}
\end{document}